\documentclass[final]{siamltex}

\usepackage{graphicx}
\usepackage{mathptmx}      
\usepackage{amssymb,amsmath, amsfonts}
\usepackage{a4wide}
\newtheorem{example}{Example}
\newtheorem{remark}{Remark}

\newcommand{\email}[1]{{\tt #1}}

\newcommand{\Gr}{{\rm gph\,}}

\newcommand{\co}{{\rm conv\,}}

\newcommand{\xb}{\bar x}
\newcommand{\yb}{\bar y}

\newcommand{\R}{\mathbb{R}}
\newcommand{\norm}[1]{\|#1\|}

\newcommand{\dist}[1]{{\rm d}(#1)}
\newcommand{\B}{{\cal B}}

\newcommand{\mv}{\,\vert\, }
\newcommand{\setto}[1]{\mathop{\to}\limits^#1}

\newcommand{\AT}[2]{{\textstyle{#1\atop#2}}}

\newcommand{\It}[1]{\par\noindent\hspace*{#1\parindent}\ignorespaces}

\newcommand{\Itl}[2]{\par\noindent\makebox[#1\parindent][l]{#2}\ignorespaces}
\newcommand{\Itlb}[2]{\par\noindent\makebox[#1\parindent][l]{#2}\hspace*{-\parindent}\makebox[\parindent][l]{\{}\ignorespaces}

\begin{document}

\title{Optimality conditions for disjunctive programs based on generalized differentiation with application to mathematical programs with equilibrium constraints\thanks{This work was supported by the Austrian Science Fund (FWF) under grant P 26132-N25}.}

\author{Helmut Gfrerer\thanks{Institute of Computational Mathematics, Johannes Kepler University Linz,
              A-4040 Linz, Austria,
             \email{helmut.gfrerer@jku.at}}
}

\date{}

\maketitle

\begin{abstract}We consider optimization problems with a disjunctive structure of the constraints. Prominent examples of such problems are mathematical programs with equilibrium constraints or vanishing constraints. Based on the concepts of directional subregularity and their characterization by means of objects from generalized differentiation, we obtain the new stationarity concept of extended M-stationarity, which turns out to be an equivalent dual characterization of  B-stationarity. These results are valid under a very weak constraint qualification of Guignard type which is usually very difficult to verify. We also state a new constraint qualification which is a little bit stronger but verifiable. Further we present second-order optimality conditions, both necessary and sufficient. Finally we apply these results to the special case of mathematical programs with equilibrium constraints and compute explicitly all the objects from generalized differentiation.  For this type of problems we also introduce the concept of strong M-stationarity which builds a bridge between S-stationarity and M-stationarity.
\end{abstract}
\begin{keywords}Optimality conditions, M-stationarity, metric subregularity
\end{keywords}
\begin{AMS}
49J53 \and 49K27 \and 90C48
\end{AMS}
\section{Introduction}
In this paper we consider mathematical programs  with disjunctive constraints of the form
\begin{equation}
\label{EqMPDC}\min_{x\in\R^n}f(x)\quad\text{subject to }\quad F(x)\in \Omega,
\end{equation}
where the functions $f:\R^n\to\R$, $F:\R^n\to\R^m$ are continuously differentiable and $\Omega=\cup_{i=1}^{\bar p} P_i$ is the union of finitely many convex polyhedra $P_i$.

A prominent example for such programs are mathematical programs
with equilibrium constraints (MPEC for short)
\begin{eqnarray}
\nonumber\min_{x\in\R^n}&&f(x)\\
\label{EqMPEC}\text{subject to }&&g(x)\leq 0,\\
\nonumber &&h(x)=0,\\
\nonumber &&G_i(x)\geq 0, H_i(x)\geq 0, G_i(x)H_i(x)=0,\ i=1,\ldots,q
\end{eqnarray}
with functions $f:\R^n\to\R$, $g:\R^n\to \R^l$, $h:\R^n\to\R^p$, and
$G, H: \R^n\to \R^q$. Note that the complementarity conditions $G_i(x)\geq 0, H_i(x)\geq 0, G_i(x)H_i(x)=0$
can be equivalently rewritten in the form
\begin{equation}
\label{EqComplCond}
-(G_i(x),H_i(x))\in Q_{\rm EC},\ i=1,\ldots,q,
\end{equation}
where
\begin{equation}\label{EqQEC}Q_{\rm EC}:=\{(a,b)\in\R^2_-\mv ab=0\}
\end{equation}
is the union of the 2 polyhedral sets $\R_-\times\{0\}$ and $\{0\}\times\R_-$. Hence the MPEC is of the form \eqref{EqMPDC} with
\begin{equation}\label{EQMPECData}F(x)=\left(\begin{array}{c}g(x)\\h(x)\\
-G_1(x)\\-H_1(x)\\\vdots\\-G_q(x)\\-H_q(x)\end{array}\right),\ \Omega=\R_-^l\times\{0\}^p\times Q_{\rm EC}^q
\end{equation}
is the union of $2^q$ polyhedral sets. Here, the minus signs are used only for convenience of the subsequent analysis.

MPECs have their origin in bilevel programming and  arise in many
applications in economic, engineering and natural sciences. We
refer to the monographs \cite{LuPaRa96, OutKoZo98} for further
details.

MPECs are known to be difficult optimization problems because  due
to the complementarity conditions $G_i(x)\geq 0,\ H_i(x)\geq 0,\
G_i(x)H_i(x)=0$ many of the standard constraint qualifications  of
nonlinear programming are violated at any feasible point. Hence
the usual Karush-Kuhn-Tucker conditions fail to hold at a local
minimzer and various first-order optimality conditions such as
Abadie (A-), Bouligand (B-), Clarke (C-), Mordukhovich (M-) and
Strong (S-) stationarity conditions have been studied in the
literature \cite{FleKan03, FuPa99,
KaSchw10a,Out00,Out99,SchSch00,Ye00,Ye05,YeYe97}. B-stationarity
expresses the  first-order necessary condition that there does not
exist a feasible descent direction at a local optimum and this is
actually the type of stationarity we want to characterize.
However, B-stationarity is  very difficult to verify  because it
is a primal first-order condition. Hence the other stationarity
concepts, which are dual first-order conditions, have been
introduced. S-stationarity is sufficient for B-stationarity but it
only holds  under some strong constraint qualification such as
MPEC-LICQ (Linear Independence Constraint Qualification). A
slightly weaker stationary concept is M-stationarity which holds
under fairly mild assumptions. However, it is to be noted that
M-stationarity (and therefore also the weaker concepts of A- and
C-stationarity) does not preclude the occurrence of feasible
descent directions. Till now no dual first-order condition is
known which is equivalent to B-stationarity under some weak
constraint qualification and we will close this gap in this paper.

Compared with the first-order optimality conditions, very little
has been done with the second-order optimality conditions for
MPECs. In \cite{SchSch00} necessary and sufficient conditions
based on the concept of S-stationarity have been stated. In
\cite{GuoLinYe12b} second-order necessary conditions in terms of
S- and C-multipliers were stated and some consequences of a strong
second-order sufficient conditions based on M-multipliers were
given.

Another example for programs with disjunctive constraints arise from programs with vanishing constraints (MPVC)
\begin{equation}
\label{EqVC}H_i(x)\geq 0,\ G_i(x)H_i(x)\leq 0, i=1,\ldots,q
\end{equation}
which can be equivalently formulated as
\[(G_i(x),H_i(x))\in Q_{\rm VC}\]
with $Q_{\rm VC}$ being the union of the 2 polyhedral sets $\R_-\times\R_+$ and $\R\times\{0\}$. For more details on MPVCs and optimality conditions we refer the reader to \cite{AchHoKa08,AchKa08,HoKa08,HoKa09,HoKaOut10,IzSo09}.

For S- and M-stationarity conditions for mathematical programs
with disjunctive constraints we refer to \cite{FleKanOut07}.

The aim of this paper is to present a unified theory of optimality
conditions based on the concepts of generalized differentiation by
Mordukhovich \cite{Mo06a,Mo06b}. In fact, by the Mordukhovich
criterion \cite[Theorem 4.18]{Mo06a}, the M-stationarity
conditions  state that a certain multifunction built by the
objective and the constraints is not metrically regular. Our
optimality conditions rely on the observation that at a local
minimizer for every critical direction such a multifunction cannot
have a certain regular behaviour. They are obtained by applying
the characterizations of directional metric regularity,
subregularity, and mixed regularity/subregularity  as can be found
in the recent papers \cite{Gfr11,Gfr13a,Gfr13b}. The resulting
first-order and second-order optimality conditions are of the form
that for every critical direction there is some multiplier
fulfilling the optimality condition. Recall  that  the standard
second-order necessary optimality conditions  for a nonlinear
programming problem, see e.g. \cite{Ben80, Iof79c,LeMiOs78}, have the same structure, namely that for
every critical direction there is some multiplier fulfilling the
second-order condition. In the context of disjunctive programming
we now need this directional form also for the first-order
conditions in order to obtain  strong necessary conditions.

In section 2 we recall the basic  definitions of the different
versions of regularity and their characterizations by means of
generalized differentiation. In section 3 we state various
optimality conditions for the problem \eqref{EqMPDC}. We obtain
first-order optimality conditions called {\em extended
M-stationarity conditions} and we will show that under some weak
constraint qualification this condition is an equivalent dual
characterization of B-stationarity.

Further we  introduce a  new constraint qualification based on
directional metric subregularity which appears to be rather weak
but is verifiable. Finally, second-order optimality conditions,
both necessary and sufficient are presented.

In section 4 we apply these results  to MPECs by explicitly
calculating the objects from generalized differentiation. Since
extended M-stationarity is still difficult to verify, we present
also the  weaker necessary condition of {\em strong
M-stationarity} which builds a bridge between S- and
M-stationarity and seems to be well suited for numerical purposes.

In what follows we denote by  $\B_{\R^n}:=\{x\in \R^n\mv
\norm{x}\leq 1\}$ the closed unit ball. For a mapping
$F:\R^n\to\R^m$ we denote by $\nabla F(\xb)$ the Jacobian and by
$\nabla^2F(\xb)$ the second derivative as defined by
\[u^T\nabla^2 F(\xb):=\lim_{t\to 0}\frac{\nabla F(\xb+tu)-\nabla F(\xb)}t\ \forall u\in\R^n.\]
Hence, for a scalar  mapping $f:\R^n\to\R$, $\nabla^2 f(\xb)$ can
be identified with the Hessian and for a mapping $F:\R^n\to\R^m$
we have
\[u^T\nabla^2 F(\xb)v=(u^T\nabla^2 F_1(\xb)v,\ldots,u^T\nabla^2 F_m(\xb)v)^T.\]

\section{Preliminaries}
We start by recalling several definitions and results from variational analysis: Let $\Omega\subset\R^n$ be an arbitrary closed set and $x\in\Omega$.
The {\em contingent}  (also called  {\em Bouligand} or {\em tangent}) {\em cone} to $\Omega$ at $x$, denoted by $T(x;\Omega)$, is
given by
\[T(x;\Omega):=\{u\in \R^n\mv \exists (u_k)\to u, (t_k)\downarrow 0: x+t_ku_k\in\Omega\}.\]
We denote by
\begin{equation}\label{DefEpsNormals}
\hat N(x;\Omega)=\{\xi \in\R^n\mv \limsup_{x'\setto \Omega
x}\frac{\xi^T( x'-x)}{\norm{x'-x}}\leq 0\}
\end{equation}
the {\em Fr\'echet} ({\em regular}) {\em normal cone} to $\Omega$.
Finally, the {\em Mordukhovich} ({\em basic/limiting}) {\em normal cone} to $\Omega$ at $x$ is
defined by
\[N(x;\Omega):=\{\xi \mv \exists  (x_k)\setto{\Omega}x,\ (\xi_k)\to \xi: \xi_k\in \hat N(x_k;\Omega) \forall k\}.\]
If $x\not\in \Omega$ we put $T(x;\Omega)=\emptyset$, $\hat N(x;\Omega)=\emptyset$ and $N(x;\Omega)=\emptyset$.

The Mordukhovich normal cone is generally nonconvex whereas the
Fr\'echet normal cone is always convex. In the case of a convex
set $\Omega$, both the Fr\'echet normal cone and the Mordukhovich normal cone
coincide with the standard normal cone from convex analysis and
moreover, the contingent cone is equal to the tangent cone in the
sense of convex analysis.

Note that $\xi\in\hat N(x;\Omega)$ $\Leftrightarrow$ $\xi^Tu\leq 0\ \forall u\in T(x;\Omega)$, i.e. $\hat N(x;\Omega)=\hat N(0;T(x;\Omega))=T(x;\Omega)^\circ$ is the polar cone  of $T(x;\Omega)$.

If $\Omega$ is a convex cone then we have
\[T(x;\Omega)=\Omega +\R\{x\},\quad \hat N(x;\Omega)=\{\xi\in \hat N(0;\Omega)=\Omega^\circ\mv \xi^Tx=0\},\]
whereas in case of  an arbitrary cone (not necessarily convex) we still have $\xi^T x=0$ $\forall \xi\in\hat N(x;\Omega)$ and consequently also $\xi^Tx=0$   $\forall \xi\in N(x;\Omega)$.

If $\Omega$ is the union of finitely many sets $\Omega_i$, $i=1,\ldots,\bar p$, then
\[T(x;\Omega)=\bigcup_{i:x\in\Omega_i}T(x;\Omega_i),\quad \hat N(x;\Omega)=\bigcap_{i:x\in\Omega_i}\hat N(x;\Omega_i).\]
The contingent cone and the F\'echet normal cone to a convex polyhedron $P$ with representation $P=\{x\in\R^n\mv a_j^Tx\leq b_j,\ j=1,\ldots,m\}$ are given
by
\[T(x;P)=\{u\in\R^n\mv a_j^Tu\leq 0,\ j\in{\cal A}(x)\},\quad \hat N(x;P)=\{\sum_{j\in{\cal A}(x)}\mu_ja_j\mv \mu_j\geq0,\ j\in{\cal A}(x)\},\]
where ${\cal A}(x):=\{j\mv a_j^Tx=b_j\}$ denotes the index set of active constraints.

Given a multifunction $M:\R^n\rightrightarrows \R^m$ and a point $(\xb,\yb)\in \Gr M:=\{(x,y)\in X\times Y\mv y\in M(x)\}$ from its graph, the {\em coderivative} of $M$ at $(\xb,\yb)$ is a multifunction $D^\ast M(\xb,\yb):\R^m\rightrightarrows \R^n$ with the values $D^\ast M(\xb,\yb)(\eta):=\{\xi\in \R^n\mv (\xi,-\eta)\in N((\xb,\yb);\Gr M)\}$, i.e. $D^\ast M(\xb,\yb)(\eta)$ is the collection of all $\xi\in \R^n$ for which there are sequences  $(x_k,y_k)\to (\xb,\yb)$ and $(\xi_k,\eta_k)\to(u,v)$ with $(\xi_k, -\eta_k)\in \hat N((x_k,y_k); \Gr M)$.

For more details we refer to the monographs \cite{Mo06a,RoWe98}

The following directional versions of these limiting constructions were introduced in \cite{Gfr13a}. Given a direction $u\in \R^n$, the Mordukhovich normal cone to a subset $\Omega\subset \R^n$ in direction $u$ at $x\in\Omega$ is
defined by
\[N(x;\Omega;u):=\{\xi\in\R^n\mv \exists (t_k)\downarrow 0,\ (u_k)\to u,\ (\xi_k)\to \xi: \xi_k\in \hat N(x+t_ku_k;\Omega) \forall k\}.\]
For a multifunction $M:\R^n\rightrightarrows \R^m$ and a direction $(u,v)\in \R^n\times \R^m$, the coderivative  of $M$ in direction $(u,v)$ at $(\xb,\yb)\in \Gr M$ is defined as the multifunction $D^\ast M((\xb,\yb);(u,v)):\R^m\rightrightarrows\R^n$ given by $D^\ast M((\xb,\yb);(u,v))(\eta):=\{\xi\in \R^n\mv (\xi,-\eta)\in N((\xb,\yb);\Gr M; (u,v))\}$.

Note that the directional versions of the Mordukhovich normal cone and the coderivative as defined in \cite{Gfr13a}  were introduced for general Banach spaces and therefore look somewhat different. In particular, in \cite{Gfr13a} it was distinguished between normal, mixed and reversed mixed coderivatives. However, in finite dimensional spaces weak-$^\ast$ and strong convergence coincide and hence this distinction is superfluous in our setting. In fact the definitions above are equivalent with the definitions from \cite{Gfr13a}.

Note that by the definition we have $N(x;\Omega;0)=N(x;\Omega)$ and $D^\ast M((\xb,\yb);(0,0))=D^\ast M(\xb,\yb)$. Further $N(x;\Omega;u)\subset N(x;\Omega)$ for all $u$ and $N(x;\Omega;u)=\emptyset$ if $u\not\in T(x;\Omega)$.

The following two lemmas give  characterizations of the directional Mordukhovich normal cone.

\begin{lemma}\label{LemInclNormalCone}Let $\Omega\subset\R^n$ be the union of finitely many closed convex sets $\Omega_i$, $i=1,\ldots, \bar p$, $\xb\in\Omega$, $u\in\R^n$. Then
\begin{equation}\label{EqInclNormalCone}
N(\xb;\Omega;u)\subset\{\xi\in N(\xb;\Omega)\mv \xi^Tu=0\}
\end{equation}
and
\begin{equation}
\label{EqInclRegNormalCone} \hat N(u;T(\xb;\Omega))\supset
\{\xi\in \hat N(\xb;\Omega)\mv \xi^T u=0\}
\end{equation}
If $\Omega$ is convex and $u\in T(\xb;\Omega)$ then both inclusion
hold with equality and therefore
$N(\xb;\Omega;u)=N(u;T(\xb;\Omega))=\hat N(u;T(\xb;\Omega))$.
\end{lemma}
\begin{proof}
Let $\xi\in N(\xb;\Omega;u)$ be arbitrarily  fixed and consider
sequences $(t_k)\downarrow 0$, $(u_k)\to u$, $(\xi_k)\to \xi$ with
$\xi_k\in \hat N(\xb+t_ku_k;\Omega)$ for all $k$. Let
$I_k:=\{i\in\{1,\ldots,\bar p\}\mv \xb+ t_ku_k\in \Omega_i\}$.
Since there are only finitely many subsets of $\{1,\ldots,\bar
p\}$, by passing to a subsequence  we can assume that there is
some index set $I$ such that $I_k=I$ for all $k$. Let $\bar i\in
I$ arbitrarily fixed. Since $\Omega_{\bar i}$ is closed we have
$\xb\in \Omega_{\bar i}$ and therefore
$\xi_k^T(\xb-(\xb+t_ku_k))=-t_k\xi_k^Tu_k\leq 0$, implying $\xi^T
u\geq 0$. Since $(t_k)\downarrow 0$, we can find for each $k$ some
index $j(k)\geq k$ with $t_{j(k)}/t_k\leq \frac 1k$. Then for all
$k$ we have $\xi_{j(k)}^T(\xb+t_ku_k-(\xb+t_{j(k)}u_{j(k)}))\leq
0$ and therefore
\[0\geq \lim_{k\to\infty}\xi_{j(k)}^T(u_k-\frac{t_{j(k)}}{t_k} u_{j(k)})=\xi^Tu\]
also holds. Hence the inclusion \eqref{EqInclNormalCone} is shown.

To show \eqref{EqInclRegNormalCone}, consider the index sets $\bar
I:= \{i\in\{1,\ldots,\bar p\}\mv \xb\in \Omega_i\}$, $I_u:=
\{i\in\bar I\mv u\in T(\xb;\Omega_i)\}$ and $\eta\in \{\xi\in \hat
N(\xb;\Omega)\mv \xi^T u=0\}=\bigcap_{i\in \bar I}\{\xi\in \hat
N(\xb;\Omega_i)\mv \xi^Tu=0\}$.  Taking into account that
$T(\xb;\Omega_i)$ is a convex cone and $\Omega_i$ is convex for
each $i$, we have
\[\hat N(u;T(\xb;\Omega_i))=\{\eta\in \hat N(0;T(\xb;\Omega_i))\mv
\eta^Tu=0\}=\{\eta\in \hat N(\xb;\Omega_i)\mv \eta^Tu=0\}\ \forall
i\in I_u\] and therefore
\[\eta\in\{\xi\in \hat N(\xb;\Omega)\mv \xi^T u=0\}=  \bigcap_{i\in \bar I}\{\xi\in \hat N(\xb;\Omega_i)\mv
\xi^Tu=0\}\subset \bigcap_{i\in I_u}\hat N(u;T(\xb;\Omega_i))=\hat
N(u;T(\xb;\Omega))\] proving \eqref{EqInclRegNormalCone}.

To show the assertion about equality in the case of convex
$\Omega$, let $u\in T(\xb;\Omega)$ and $\xi\in N(\xb;\Omega)$ with
$\xi^Tu=0$ be arbitrarily fixed. Then we can find sequences
$(t_k)\downarrow 0$ and $(u_k)\to u$ with
$x_k:=\xb+t_ku_k\in\Omega$ for all $k$ and, by passing to a
subsequence if necessary, we can assume $k\norm{u-u_k}\to 0$.
Because of
\[-\xi^T(x-x_k)=-\xi^T(x-\xb)+t_k\xi^Tu_k=-\xi^T(x-\xb)+t_k\xi^T(u_k-u)\geq t_k\xi^T(u_k-u)\ \forall x\in\Omega,\]
by invoking Ekeland's variational principle, there is for every $k$ some $\tilde x_k\in \Omega$ such that $\norm{\tilde x_k-x_k}\leq kt_k\xi^T(u-u_k)$ and $\tilde x_k$ is a global minimizer of the problem
\[\min_{x\in\Omega} -\xi^T(x-x_k)+\frac 1k\norm{x-\tilde x_k}.\]
By the well known first order optimality conditions from  convex
analysis (see, e.g., \cite[Theorem 27.4]{Ro70}) there is some
element $\eta_k\in\B_{\R^n}$  such that $\xi-\frac
1k\eta_k=:\xi_k\in \hat N(\tilde x_k;\Omega)$ , showing
$\lim_{k\to\infty} \xi_k=\xi$.  Since we also have
\[\limsup_{k\to\infty}\norm{\frac {\tilde x_k-\xb}{t_k}-u}\leq \limsup_{k\to\infty}(\norm{\frac {\tilde x_k-x_k}{t_k}}+\norm{\frac {x_k-\xb}{t_k}-u})\leq \limsup_{k\to\infty}(k\norm{\xi}\norm{u_k-u}+\norm{u_k-u})=0,\]
$\xi\in N(x;\Omega;u)$ follows. Finally note, that equality in
\eqref{EqInclRegNormalCone} holds for convex $\Omega$ because of
\[\hat N(u;T(\xb;\Omega))= \{\xi\in \hat N(0;T(\xb;\Omega))\mv \xi^T u=0\}
=\{\xi\in \hat N(\xb;\Omega)\mv \xi^T u=0\}.\]
\end{proof}

\begin{lemma}\label{LemBasicPropCone}Let $\Omega\subset\R^n$ be the union of finitely many polyhedra $P_i$, $i=1,\ldots, \bar p$ and let $\xb\in\Omega$ and $u\in T(\xb;\Omega)$. Then
\begin{equation}\label{EqInclNormalCone1}
N(\xb;\Omega;u)=\bigcup_{v\in  T(u;T(\xb;\Omega))}\hat N(v; T(u;T(\xb;\Omega))).
\end{equation}
\end{lemma}
\begin{proof}
Let  the polyhedra $P_i$, $i=1,\ldots,\bar p$ be represented  by
\[P_i=\{x\in\R^n \mv a_{ij}^Tx\leq b_j,\  j=1,\ldots,m_i\}\]
and denote $\bar{\cal P}:=\{i\in\{1,\ldots,\bar p\}\mv \xb\in P_i\}$ and $\bar{\cal A}_i:=\{j\in\{1,\ldots,m_i\}\mv a_{ij}^T\xb=b_i\}$ for $i\in\bar{\cal P}$.
Then
\[T(\xb;\Omega)=\bigcup_{i\in\bar{\cal P}}T(\xb;P_i)= \bigcup_{i\in\bar{\cal P}}\{v\in\R^n\mv a_{ij}^Tv\leq 0,\ j\in\bar{\cal A}_i\}\]
and denoting ${\cal P}(u):=\{i\in\bar {\cal P}\mv u\in T(\xb;P_i)\}$ and ${\cal A}_i(u):=\{j\in\bar{\cal A}_i\mv a_{ij}^Tu=0\}$ for $i\in{\cal P}(u)$ we have
\begin{equation}
\label{EqTanTanCone}
T(u;T(\xb;\Omega))=\bigcup_{i\in{\cal P}(u)}\{v\in\R^n\mv a_{ij}^T v\leq 0,\ j\in{\cal A}_i(u)\}.
\end{equation}
Now let $v\in T(u;T(\xb;\Omega))$ and $\xi \in \hat
N(v;T(u;T(\xb;\Omega)))$ be  arbitrarily fixed and let ${\cal
P}^v:=\{i\in{\cal P}(u)\mv v\in T(u;T(\xb;P_i))\}$ and ${\cal
A}_i^v:=\{j\in{\cal A}_i(u)\mv a_{ij}^Tv=0\}$, $i\in{\cal P}^v$.
Then $\xi\in \bigcap_{i\in{\cal P}^v}\hat N(v;T(u;T(\xb;P_i)))$
and thus for each $i\in{\cal P}^v$ there are nonnegative numbers
$\mu_{ij}\geq 0$, $j\in{\cal A}^v_i$ such that
$\xi=\sum_{j\in{\cal A}_i^v}\mu_{ij}a_{ij}$. We claim that for all
$t>0$ sufficiently small $\xi\in \hat N(\xb+tu+t^2v;\Omega)$. To
prove this claim it is sufficient to show ${\cal
P}^v=\{i\in\{1,\ldots,\bar p\}\mv \xb+tu+t^2 v\in P_i\}$ and
${\cal A}_i^v=\{j\in\{1,\ldots,m_i\}\mv
a_{ij}^T(\xb+tu+t^2v)=b_j\}$, $i\in{\cal P}^v$ for all $t>0$
sufficiently small, since then
\[\hat N(\xb+tu+t^2v;\Omega)=\bigcap_{i\in{\cal P}^v}\{\sum_{j\in{\cal A}_i^v}\mu_{ij}a_{ij}\mv \mu_{ij}\geq 0, j\in{\cal A}_i^v\}.\]
Let $i\in{\cal P}^v$ and $j\in\{1,\ldots,m_i\}$. The index set $\{1,\ldots,m_i\}$ can be partitioned into the four sets $J_1=\{1,\ldots,m_i\}\setminus \bar{\cal A}_i$, $J_2:=\bar{\cal A}_i\setminus{\cal A}_i(u)$, $J_3:={\cal A}_i(u)\setminus {\cal A}_i^v$ and ${\cal A}_i^v$. If $j\in {\cal A}_i^v$ we have $a_{ij}^Tv=a_{ij}^Tu=a_{ij}\xb-b_j=0$, if $j\in J_3$ we have $a_{ij}^Tv<a_{ij}^Tu=a_{ij}\xb-b_j=0$, if $j\in J_2$ we have $a_{ij}^Tu<a_{ij}\xb-b_j=0$ and finally $a_{ij}\xb-b_j<0$ for $j\in J_1$. This shows $a_{ij}^T(\xb+tu+t^2v)-b_j=0$, $j\in {\cal A}^v_i$ and $a_{ij}^T(\xb+tu+t^2v)-b_j<0$, $j\in J_1\cup J_2\cup J_3$ for all $t>0$ sufficiently small and we can conclude ${\cal P}^v\subset\{i\in\{1,\ldots,\bar p\}\mv \xb+tu+t^2 v\in P_i\}$ and ${\cal A}_i^v=\{j\in\{1,\ldots,m_i\}\mv a_{ij}^T(\xb+tu+t^2v)=b_j\}$, $i\in{\cal P}^v$. For $i\not\in {\cal P}^v$ we either have $i\in I_1:=\{1,\ldots,\bar p\}\setminus \bar{\cal P}$ or $i\in I_2:=\bar{\cal P}\setminus{\cal P}(u)$ or $i\in I_3:={\cal P}(u)\setminus{\cal P}^v$. If $i\in I_1$ there is some $j\in\{1,\ldots,m_i\}$ with $a_{ij}^T\xb-b_j>0$, if $i\in I_2$ there is some $j\in\bar{\cal A}_i$ with $0=a_{ij}^T\xb-b_j<a_{ij}^Tu$ and finally for $i\in I_3$ there is some $j\in{\cal A}_i(u)$ with $0=a_{ij}^T\xb-b_j=a_{ij}^Tu<a_{ij}^Tv$. Hence there is some $j$ with $a_{ij}^T(\xb+tu+t^2v)>0$ for all $t>0$ sufficiently small and ${\cal P}^v\supset\{i\in\{1,\ldots,\bar p\}\mv \xb+tu+t^2 v\in P_i\}$ follows and our claim is proved. Since $\lim_{t\downarrow0} t^{-1}(\xb+tu+t^2v-\xb)=u$ we obtain $\xi\in N(\xb;\Omega;u)$ and $\bigcup_{v\in  T(u;T(\xb;\Omega))}\hat N(v; T(u;T(\xb;\Omega)))\subset N(\xb;\Omega;u)$ follows.

To show the reverse inclusion, let $\xi\in N(\xb;\Omega;u)$  and
consider sequences $(t_k)\downarrow 0$, $(u_k)\to u$ and
$(\xi_k)\to \xi$ such that $\xi_k\in \hat N(\xb+t_ku_k;\Omega)$.
Then for all $k$ sufficiently large we have ${\cal P}^k:=\{i\in
\{1,\ldots,\bar p\}\mv \xb+t_ku_k\in P_i\}\subset {\cal P}(u)$ and
${\cal A}_i^k:=\{j\in\{1,\ldots,m_i\}\mv
a_{ij}^T(\xb+t_ku_k)=b_i\}=\{j\in \bar{\cal A}_i\mv
a_{ij}^Tu_k=0\}\subset {\cal A}_i(u)$, $i\in{\cal P}^k$  and
\[\xi_k\in \bigcap_{i\in{\cal P}^k}\hat N(\xb+t_ku_k;P_i)=\bigcap_{i\in{\cal P}^k}\{\sum_{j\in{\cal A}_i^k}\mu_{ij}a_{ij}\mv \mu_{ij}\geq0\}.\]
It follows that $a_{ij}^Tu_k\leq 0$,  $j\in{\cal A}_i(u)$,
$i\in{\cal P}^k$ and hence $u_k\in T(u; T(\xb;P_i))$, $i\in {\cal
P}^k$. Since there are only finitely many subsets of
$\{1,\ldots,\bar p\}$ and $\{1,\ldots,m_i\}$, $i\in\{1,\ldots,\bar
p\}$ we can assume, by eventually passing to a subsequence, that
there are index sets ${\cal P}^\xi\subset{\cal P}(u)$, ${\cal
A}_i^\xi\subset {\cal A}_i(u)$, $i\in{\cal P}^\xi$ such that
${\cal P}^k={\cal P}^\xi$, ${\cal A}_i^k={\cal A}_i^\xi$,
$i\in{\cal P}^\xi$ for all $k$. Because the normal cones $\hat
N(\xb+t_ku_k;P_i)$, $i\in{\cal P}^\xi$ are closed, we obtain
$\xi\in\hat N (\xb+t_ku_k;P_i)$, $i\in{\cal P}^\xi$. Now let $k$
be arbitrarily fixed. For every $i\in{\cal P}(u)\setminus {\cal
P}^\xi$ there is some index $j_i\in{\cal A}_i(u)$ with
$a_{ij_i}^Tu_k>0$ and therefore $u_k\not\in T(u; T(\xb;P_i))$.
Since $T(u; T(\xb;P_i))$ is closed we can find $\delta>0$ such
that for every $w\in T(u;T(\xb;\Omega))\cap(u_k+\delta\B_{\R^n})$
we have $w\not\in T(u; T(\xb;P_i))$, $i\in {\cal P}(u)\setminus
{\cal P}^\xi$ showing ${\cal P}^w:=\{i\in {\cal P}(u)\mv w\in
T(u;T(\xb;P_i))\}\subset {\cal P}^\xi$. Thus, for every $i\in{\cal
P}^w$ there are nonnegative numbers $\mu_{ij}\geq 0$, $j\in{\cal
A}_i^\xi$ with $\xi=\sum_{j\in{\cal A}_i^\xi}\mu_{ij}a_{ij}$
implying
\[\xi^T(w-u_k)=\sum_{j\in{\cal A}_i^\xi}\mu_{ij}a_{ij}^T(w-u_k)=\sum_{j\in{\cal A}_i^\xi}\mu_{ij}a_{ij}^Tw\leq 0\]
because of $a_{ij}^Tw\leq 0$, $j\in{\cal A}_i(u)\supset{\cal A}_i^\xi$ and we conclude $\xi\in \hat N(u_k,T(u;T(\xb;\Omega)))$. This finishes the proof.
\end{proof}

In particular it follows from \eqref{EqInclNormalCone1} that for every $\bar v\in T(u;T(\xb;\Omega))$ we have
\begin{eqnarray}
\label{EqInclNormalCone2}&\hat N(\bar v; T(u;T(\xb;\Omega)))\subset N(\xb;\Omega;u),&\\
\label{EqInclNormalCone3} &N(\bar v; T(u;T(\xb;\Omega)))=\limsup_{v\to\bar v}\hat N(v; T(u;T(\xb;\Omega)))\subset N(\xb;\Omega;u).&
 \end{eqnarray}

In what follows we consider the notions of metric regularity and subregularity, respectively, and its characterization by coderivatives and Mordukhovich normal cones.

Recall that a multifunction $M:\R^n\rightrightarrows \R^m$ is called {\em
metrically regular} with modulus $\kappa>0$ near the point $(\xb,\yb)\in \Gr M$ from its graph,
provided there exist  neighborhoods $U$ of $\xb$ and $V$ of $\yb$
such that
\begin{equation}
\label{EqReg0} \dist{x,M^{-1}(y)}\leq \kappa \dist{y,M(x)}\quad \forall (x,y)\in U\times V.
\end{equation}
Here $\dist{x,\Omega}$ denotes the usual distance between a point $x$ and a set $\Omega$.

It is well known that metric regularity of the multifunction $M$ near $(\xb,\yb)$ is equivalent to the Aubin property of the inverse multifunction $M^{-1}$. A multifunction $S:\R^m\rightrightarrows \R^n$ has the {\em Aubin property} with modulus $L\geq 0$ near some point $(\yb,\xb)\in\Gr S$, if there are neighborhoods $U$ of $\xb$ and $V$ of $\yb$ such that
\[S(y_1)\cap U\subset S(y_2)+L\norm{y_1-y_2}\B_{\R^n}\quad \forall y_1,y_2\in V.\]
We refer to the monographs \cite{Mo06a,Mo06b,KlKum02,RoWe98} and the survey  \cite{Iof00} for an extensive treatment of these subjects and the related notions of  {\em pseudo-Lipschitz continuity}, {\em Lipschitz-like property} and {\em openness with a linear rate}.

Metric regularity can be equivalently characterized by the so-called {\em Mordukhovich criterion} (cf. \cite[Theorem 4.18]{Mo06a}, \cite[Theorem 9.43]{RoWe98}):

\begin{theorem}For a multifunction $M:\R^n\rightrightarrows\R^m$ with closed graph and any $(\xb,\yb)\in\Gr M$ the following statements are equivalent:
\begin{enumerate}
\item $M$ is metrically regular near $(\xb,\yb)$.
\item $\ker D^\ast M(\xb,\yb)=\{0\}$, i.e. $0\in D^\ast M(\xb,\yb)(\lambda)\Rightarrow \lambda=0$.
\end{enumerate}
\end{theorem}
Applying this criterion to multifunctions of the form $M(x)=F(x)-\Omega$ we obtain the following collorary, see e.g. \cite{RoWe98}:
\begin{corollary}
Let $M:\R^n\rightrightarrows\R^m$, $M(x)=F(x)-\Omega$ be a multifunction, where $F:\R^n\to\R^m$ is continuously differentiable, $\Omega\subset\R^m$ is closed and let $\xb\in\R^n$ be given with $F(\xb)\in\Omega$. Then $M$ is metrically regular near $(\xb,0)$ if and only if
\begin{equation}\label{EqMCCQ}
\nabla F(\xb)^T\lambda=0,\ \lambda\in N(F(\xb);\Omega)\ \Rightarrow \lambda=0
\end{equation}
\end{corollary}

Among other things metric regularity is important in the context of constraint qualifications:

\begin{example}Consider a system of inequalities  and equalities
\[ g(x)\leq 0,\ h(x)=0\]
with continuously differentiable functions $g:\R^n\to\R^l$, $h:\R^n\to \R^p$. Recall that at a solution $\xb$ the {\em Mangasarian Fromovitz constraint qualification} (MFCQ) is said to hold, if the gradients $\nabla h_i(\xb)$ are linearly independent and there exists a vector $z\in\R^n$ with
\[\nabla h(\xb)z=0, \nabla g_i(\xb) z<0,\ i\in I(\xb),\]
where $I(\xb)=\{i\mv g_i(\xb)=0\}$ denotes the index set of active inequalities.

It is well known \cite{Rob76a} that MFCQ is equivalent with metric regularity of the multifunction $M:\R^n\rightrightarrows \R^l\times \R^p$
\[M(x):=(g(x),h(x))-\R_-^l\times \{0\}^p\]
near $(\xb,0)$. Straightforward calculations yield that condition \eqref{EqMCCQ} reads as
\[\sum_{i\in I(\xb)}\nabla g_i(\xb)\lambda^g_i+\sum_{i=1}^p\nabla h_i(\xb)\lambda^h_i=0, \lambda^g_i\geq 0,i\in I(\xb)\ \Rightarrow \lambda^g_i=0, i\in I(\xb), \lambda^h_i=0,\ i=1,\ldots,p\]
which is also called {\em positive linear independence constraint qualification}.
\end{example}

Condition \eqref{EqMCCQ} appears under different names in the literature. E.g. in \cite{GuoLinYe12a}, \cite{Ye00}, \cite{Ye05} it is called {\em no nonzero abnormal multiplier constraint qualification} (NNAMCQ), whereas in \cite{FleKanOut07} it is called {\em generalized Mangasarian-Fromovitz constraint qualification} (GMFCQ).

When fixing $y=\yb$ in \eqref{EqReg0} we obtain the weaker property of {\em metric subregularity} of $M$ at $(\xb,\yb)$, i.e. we require the estimate
\begin{equation}
\label{EqSubReg0} \dist{x,M^{-1}(\yb)}\leq \kappa \dist{\yb,M(x)}\quad \forall x\in U
\end{equation}
with some neighborhood $U$of $\xb$ and a positive real $\kappa>0$.

The metric subregularity property was introduced by Ioffe
\cite{Iof79a, Iof00} using the terminology ''regularity at a
point''. The notation ''metric subregularity'' was suggested in
\cite{DoRo04}.  It is well known
\cite{DoRo04} that metric subregularity of $M$ at $(\xb,\yb)$ is
equivalent to {\em calmness} of the inverse multifunction $M^{-1}$
at $(\yb,\xb)$. It seems to be that the concept of calmness of a
set-valued map first appear in Ye and Ye \cite{YeYe97} under the
term ''pseudo upper-Lipschitz continuty''. Criteria for subregularity and calmness,
respectively, can be found e.g. in the papers
\cite{FabHenKrOut10, Gfr11, Gfr14a, HenJou02, HenJouOut02, HenOut01, HenOut05, IofOut08, Kum09, ZhNg07, ZhNg10}. An important subclass of
multifunctions which are known to be metrically subregular  at
every point of its graph, is given by polyhedral multifunctions,
i.e. multifunctions whose graph is the union of finitely many
polyhedral sets. This result is due to Robinson \cite{Rob81}.  An
important special case of polyhedral multifunctions is given by
linear systems, where subregularity is a consequence of Hoffman's
error bound \cite{Hoff52}, whereas, as pointed out in the example above, metric regularity is equivalent to MFCQ.

We consider also the following concept of mixed metric regularity/subregularity for multifunctions $M$ composed by two multifunctions $M_i:\R^n\rightrightarrows \R^{m_i}$, $i=1,2$, i.e. $M$ has the form
\[M=(M_1,M_2):\R^n\to\R^{m_1}\times\R^{m_2},\ M(x)=M_1(x)\times M_2(x).\]
We say that $M=(M_1,M_2)$ is {\em mixed metrically regular/subregular} at a point $(\xb,(\yb_1,\yb_2))\in \Gr M$, if there are neighborhoods $U$ of $\xb$ and $V_1$ of $\yb_1$ such that
\[\dist{x,M^{-1}(y_1,\yb_2)}\leq \kappa \dist{(y_1,\bar y_2),M(x)}\quad \forall (x,y_1)\in U\times V_1.\]
Clearly, mixed metric regularity/subregularity of $(M_1,M_2)$ implies metric subregularity of $M$.
\begin{theorem}\label{ThMixedRegSubreg}
Let $M_i:\R^n\rightrightarrows\R^{m_i}$, $M_i(x)=F_i(x)-\Omega_i$, $i=1,2$ be two multifunctions, where $F_i:\R^n\to\R^{m_i}$ is continuously differentiable and $\Omega_i\subset\R^{m_i}$ is the union of finitely many convex polyhedra and let $F_i(\xb)\in\Omega_i$. Assume that $M_2$ is metrically subregular at $(\xb,0)$ and that
\[\nabla F_1(\xb)^T\lambda_1+\nabla F_2(\xb)^T\lambda_2=0,\ \lambda_i\in N(F_i(\xb);\Omega_i), i=1,2 \ \Rightarrow\ \lambda_1=0.\]
Then the multifunction $M=(M_1,M_2)$ is mixed regular/subregular at $(\xb,(0,0))$
\end{theorem}
\begin{proof}
Consider the multifunction $S:\R^n\rightrightarrows \R^{m_1}\times\R^{m_2}$, $S(x)=(-\Omega_1)\times(-\Omega_2)$. Since $F_1$ is  continuously differentiable, it is also Lipschitz near $\xb$ and therefore $M_1$ has the Aubin property near $(\xb,0)$. In finite dimensions every linear operator has closed range. Hence we can invoke \cite[Lemma 2.4]{Gfr13b} together with \cite[Theorem 4.3]{Gfr13b} to obtain that the condition
\[0\in \nabla F_1(\xb)^T\lambda_1+\nabla F_2(\xb)^T\lambda_2 +D^\ast S(\xb,(-F_1(\xb),-F_2(\xb)))(\lambda_1,\lambda_2)\ \Rightarrow\ \lambda_1=0\]
is sufficient for mixed regularity/subregularity of $M$. Since
\[D^\ast S(\xb,(-F_1(\xb),-F_2(\xb)))(\lambda_1,\lambda_2)=\begin{cases}
\{0\}&\mbox{if $\lambda_i\in N(F_i(\xb);\Omega_i)$, $i=1,2$,}\\\emptyset&\mbox{else,}
\end{cases}\]
the assertion follows.
\end{proof}

For our analysis we also need the notion of directional metric subregularity. To define this property it is convenient to introduce  the following neighborhoods of directions:
Given   a direction $u\in \R^n$ and positive numbers $\rho,\delta>0$, the set $V_{\rho,\delta}(u)$,
 is given by
\begin{equation}
\label{EqDefNbhd}V_{\rho,\delta}(u):=\{z\in\rho \B_{\R^n}\mv
\big\Vert \norm{u} z- \norm{z} u \big\Vert\leq\delta \norm{z}\
\norm{u}\}.
\end{equation}
This can also be written in the form
\[V_{\rho,\delta}(u)=\begin{cases}\{0\}\cup \big\{z\in\rho \B_{\R^n}\setminus \{0\}\mv
\left\Vert \frac z{\norm{z}} - \frac u{\norm{u}} \right\Vert\leq\delta\big\}&\mbox{if $u\not=0$,}\\
\rho \B_{\R^n} &\mbox{if $u=0$.}\end{cases}\]

Given $u\in \R^n$, the multifunction $M:\R^n\rightrightarrows\R^m$ is said to be
{\em metrically subregular in direction $u$} at $(\xb,\yb)\in\Gr M$, if there are  positive reals $\rho>0,\delta>0$  and $\kappa>0$
such that
\begin{equation}
\label{EqDirSubReg} \dist{x,M^{-1}(\yb)}\leq \kappa \dist{\yb,M(x)}
\end{equation}
holds for all $x\in \xb+V_{\rho,\delta}(u)$.

Note that metric subregularity in direction $0$ is equivalent to the property of metric subregularity.

\begin{theorem}\label{ThDirSubReg}
Let $M_i:\R^n\rightrightarrows\R^{m_i}$, $M_i(x)=F_i(x)-\Omega_i$, $i=1,2$ be two multifunctions, where $F_i$ is continuously differentiable and $\Omega_i$ is the union of finitely many convex polyhedra and let $F_i(\xb)\in\Omega_i$. Further, given $u\in\R^n$ assume that $M_2$ is metrically subregular in direction $u$ at $(\xb,0)$.
\begin{enumerate}
\item If
\[\nabla F_1(\xb)^T\lambda_1+\nabla F_2(\xb)^T\lambda_2=0,\ \lambda_i\in N(F_i(\xb);\Omega_i;\nabla F_i(\xb)u), i=1,2 \ \Rightarrow\ \lambda_1=0,\]
then the multifunction $M=(M_1,M_2)$ is metrically subregular in direction $u$ at $(\xb,(0,0))$.
\item Assume that $F_i$, $i=1,2$ are  twice Fr\'echet differentiable at $\xb$ and $u\not=0$. If
\[u^T\nabla^2(\lambda_1^T F_1+ \lambda_2^T F_2)(\xb)u<0\]
holds for all $(\lambda_1,\lambda_2)\in N(F_1(\xb);\Omega_1;\nabla F_1(\xb)u)\times N(F_2(\xb);\Omega_2;\nabla F_2(\xb)u)$ with $\nabla F_1(\xb)^T\lambda_1+\nabla F_2(\xb)^T\lambda_2=0$ and $\lambda_1\not=0$,  then the multifunction $M=(M_1,M_2)$ is metrically subregular in direction $u$ at $(\xb,(0,0))$.
\end{enumerate}
\end{theorem}
\begin{proof}
Using similar arguments as in the proof of Theorem \ref{ThMixedRegSubreg}, the assertion follows from \cite[Theorem 4.3]{Gfr13b} together with \cite[Lemma 2.4]{Gfr13b} by taking into account
\begin{eqnarray*}\lefteqn{D^\ast S((\xb,(-F_1(\xb),-F_2(\xb))); (u, (-\nabla F_1(\xb)u,-\nabla F_2(\xb)u)) )(\lambda_1,\lambda_2)}\\
&=&\begin{cases}
\{0\}&\mbox{if $\lambda_i\in N(F_i(\xb);\Omega_i;\nabla F_i(\xb)u)$, $i=1,2$,}\\\emptyset&\mbox{else,}
\end{cases}\end{eqnarray*}
\end{proof}

Characterization of directional metric subregularity also yields a characterization of metric subregularity:
\begin{lemma}
Let $M:\R^n\rightrightarrows \R^m$ be a multifunction and $(\xb,\yb)\in\Gr M$. Then $M$ is metrically subregular at $(\xb,\yb)$ if and only if it is metrically subregular in every direction $u\not=0$ at $(\xb,\yb)$.
\end{lemma}
\begin{proof}
The ''only if''-part is obviously true by the definition. We prove the if-part by contraposition. Assume that $M$ is not metrically subregular at $(\xb,\yb)$. Then we can find a sequence $(x_k)\to \xb$ satisfying $\dist{x_k,M^{-1}(\yb)}> k\dist{\yb,M(x_k)}$ for all $k$. Since $\xb\in M^{-1}(\yb)$ we conclude $x_k\not=\xb$ and therefore $u_k=\frac{x_k-\xb}{\norm{x_k-\xb}}$ is well defined. By eventually passing to a subsequence, we can assume that the sequence $(u_k)$ converges to some element $u\in\R^n$ with $\norm{u}=1$. Now let $\rho>0$ and $\delta>0$ be arbitrarily fixed. Then for all $k$ sufficiently large we have $x_k\in\xb+\rho \B_{\R^n}$ and
\[\Bigl\Vert \frac{x_k-\xb}{\norm{x_k-\xb}}-\frac u{\norm{u}}\Bigr\Vert=\norm{u_k-u}\leq\delta,\]
showing $x_k\in\xb+V_{\rho,\delta}(u)$. Hence $M$ is not metrically subregular in direction $u$.
\end{proof}

\section{Optimality conditions for the disjunctive program}

Now we apply the results of the preceding section to the problem \eqref{EqMPDC}.
We  denote the feasible region of \eqref{EqMPDC} by ${\cal F}$ and for a feasible point $\xb\in{\cal F}$ we define the {\em linearized cone } by
\[T_{\rm lin}(\xb):= \{u\in\R^n\mv \nabla F(\xb)u\in T(F(\xb);\Omega)\}\]
and the {\em cone of critical directions }  by
\[{\cal C}(\xb):=\{u\in T_{\rm lin}(\xb)\mv \nabla f(\xb) u\leq 0\}.\]
Note that always $0\in{\cal C}(\xb)$ and that $T(\xb;{\cal F})\subset T_{\rm lin}(\xb)$.

Throughout this section we assume that for every $u\in T_{\rm lin}(\xb)$ the constraint mapping $M(x)=F(x)-\Omega$ can be split into two parts $M=(M_1,M_2):\R^n\rightrightarrows \R^{m_1}\times\R^{m_2}$  with $M_i(x)=F_i(x)-\Omega_i$ and $m=m_1+m_2$,
where for each $i\in\{1,2\}$ the mapping $F_i:\R^n\to\R^{m_i}$ is continuously differentiable, $\Omega_i\subset\R^{m_i}$ is the union of finitely many convex polyhedra, $F=(F_1,F_2)$, $\Omega=\Omega_1\times\Omega_2$ and $M_2$ is  metrically subregular in direction $u$ at the  point $(\xb,0)$ with modulus $\kappa_2(u)$.

This assumption is e.g. automatically fulfilled if $F_2$ is affine linear, because then $M_2$ is a polyhedral multifunction and therefore metrically subregular at every point of its graph \cite{Rob81}. If we cannot identify some part of the multifunction which is metrically subregular in the considered direction $u$ then we can simply take $m_2=0$. Note that this splitting is not unique and to ease the notation we also suppress the dependence on $u$.

To state our optimality conditions in a general framework we consider for arbitrary  $\eta\in\R^n$ and $\xb\in {\cal F}$ the multifunction ${\cal M}^{\eta,\xb}:=({\cal M}_1^{\eta,\xb},{\cal M}_2):\R^n\rightrightarrows\R\times\R^m$ given by
\[ {\cal M}_1^{\eta,\xb}(x)=f(x)-f(\xb)+(\eta^T(x-\xb))^3-\R_-,\ {\cal M}_2(x):=M(x).\]

\begin{proposition}\label{PropOptCond}
Let $\xb$ be a local minimizer for \eqref{EqMPDC}. Then ${\cal M}^{0,\xb}$ is not mixed regular/subregular at $(\xb,0)$ and for every nonzero critical direction $0\not=u\in{\cal C}(\xb)$ there exists some $\eta$ such that ${\cal M}^{\eta,\xb}$ is not metrically subregular in direction $u$.
\end{proposition}
\begin{proof}
Follows from \cite[Proposition 5.1]{Gfr13b}.
\end{proof}

We define  the  {\em generalized Lagrangian} ${\cal L}:\R^n\times\R\times \R^{m}\to \R$ by
\[{\cal L}(x,\lambda_0,\lambda):=\lambda_0f(x)+\lambda^TF(x)\]
Given $\xb\in {\cal F}$, $u\in T_{\rm lin}(\xb)$ and $\lambda_0\geq 0$, we define the sets of multipliers
\[\Lambda^{\lambda_0}(\xb;u):=\left\{\lambda=(\lambda_1,\lambda_2)\in\R^{m_1}\times\R^{m_2}\mv \begin{array}{l}
\lambda\in N(F(\xb);\Omega;\nabla F(\xb)u),\\
\nabla_x{\cal L}(\xb,\lambda_0,\lambda)=0,\\
 \lambda_0+\norm{\lambda_1}\not=0\end{array}\right\}\]
and
\[\hat\Lambda^{\lambda_0}(\xb;u):=\left\{\lambda=(\lambda_1,\lambda_2)\in\R^{m_1}\times\R^{m_2}\mv \begin{array}{l}
\lambda\in\hat N(\nabla F(\xb)u;T(F(\xb);\Omega)),\\
\nabla_x{\cal L}(\xb,\lambda_0,\lambda)=0,\\
 \lambda_0+\norm{\lambda_1}\not=0\end{array}\right\}.\]
If $u=0$ we set
$\Lambda^{\lambda_0}(\xb):=\Lambda^{\lambda_0}(\xb;0)$,
$\hat\Lambda^{\lambda_0}(\xb):=\hat\Lambda^{\lambda_0}(\xb;0)$.

We see from the definition that the splitting $M=(M_1,M_2)$ only
influences the sets $\Lambda^0(\xb;u)$ and $\hat\Lambda^0(\xb;u)$
by the requirement that certain components of the multipliers are
not all zero.

The following lemma gives some relations between these multiplier
sets.
\begin{lemma}\label{LemRelationsMult}
For every $\lambda^0\geq 0$ and every critical direction
$u\in{\cal C}(\xb)$ we have
\[\hat\Lambda^{\lambda_0}(\xb)\subset\hat\Lambda^{\lambda_0}(\xb;u)\subset\Lambda^{\lambda_0}(\xb;u)\subset
\Lambda^{\lambda_0}(\xb)\] and equality holds, if $\bar p=1$,
i.e. $\Omega$ is a convex polyhedron.
\end{lemma}
\begin{proof}
By \eqref{EqInclNormalCone2} we have
\[\hat N(\nabla F(\xb)u; T(F(\xb);\Omega))=\hat N(0;T(\nabla F(\xb)u;T(F(\xb);\Omega)))
\subset N(F(\xb);\Omega;\nabla F(\xb)u)\] and thus
$\hat\Lambda^{\lambda_0}(\xb;u)\subset\Lambda^{\lambda_0}(\xb;u)$.
Since we also have $N(F(\xb);\Omega;\nabla F(\xb)u)\subset
N(F(\xb);\Omega)$ by the definition of the directional limiting
normal cone, the inclusion $\Lambda^{\lambda_0}(\xb;u)\subset
\Lambda^{\lambda_0}(\xb)$ immediately follows.

Now consider $\lambda\in \hat\Lambda^{\lambda_0}(\xb)$. Then
$\lambda\in T(F(\xb);\Omega)^\circ$ and hence $\lambda^T\nabla
F(\xb)u\leq 0$ because of $\nabla F(\xb)u\in T(F(\xb);\Omega)$.
From $u\in{\cal C}(\xb)$,$ \nabla_x{\cal
L}(\xb,\lambda_0,\lambda)=0$ and $\lambda^0\geq 0$ we deduce
$\lambda^T\nabla F(\xb)u =-\lambda_0\nabla f(\xb)u\geq 0$ and thus
$\lambda^T\nabla F(\xb)u=0$. Using the inclusion
\eqref{EqInclRegNormalCone} we obtain  $\lambda\in
\hat\Lambda^{\lambda_0}(\xb;u)$.

The assertion about equality follow immediately from the fact,
that for convex sets $\Omega$ the limiting and the regular normal
cone coincide and hence
$\hat\Lambda^{\lambda_0}(\xb)=\Lambda^{\lambda_0}(\xb).$
\end{proof}

The sets $\Lambda^{\lambda_0}(\xb;u)$ will be used for formulating
necessary optimality conditions, whereas $\hat
\Lambda^{\lambda_0}(\xb;u)$ plays a role when stating sufficient
conditions.

Note that $N(F(\xb);\Omega;\nabla F(\xb)u)=N(F_1(\xb);\Omega_1;\nabla F_1(\xb)u)\times N(F_2(\xb);\Omega_2;\nabla F_2(\xb)u)$ and
that for every critical direction $u$ and every $\lambda_0\geq0$ such that $\Lambda^{\lambda_0}(\xb;u)\not=\emptyset$ we have
\begin{equation}\label{EqLambda0Nabla_f}\lambda_0 \nabla f(\xb)u=-\lambda^T\nabla F(\xb)u=0\ \forall\lambda\in\Lambda^{\lambda_0}(\xb;u)\end{equation}
because of  \eqref{EqInclNormalCone}.

We are now in the position to state our main result on first-order and second-order necessary optimality conditions:

\begin{theorem}\label{ThBaseOptCond}Let $\xb$ be a local minimizer for the problem \eqref{EqMPDC} and let $u\in{\cal C}(\xb)$.  Then there exists $\lambda_0\geq 0$ such that $\Lambda^{\lambda_0}(\xb;u)\not=\emptyset$. If $f$ and $F$ are twice Fr\'echet differentiable at $\xb$ then there exist some $\lambda\in\Lambda^{\lambda_0}(\xb;u)$ with
\begin{equation}\label{EqSecOrdNec}u^T\nabla_x^2{\cal L}(\xb,\lambda_0,\lambda)u\geq 0.
\end{equation}
If $M$ is metrically subregular in direction $u$ at $(\xb,0)$ then these conditions hold with ${\lambda_0}=1$.
\end{theorem}

\begin{proof}
First consider the case $u=0$: From Proposition \ref{PropOptCond}
we know that $({\cal M}_1^{0,\xb},M)$ is not mixed
regular/subregular at $(\xb,0)$. If $M$ is subregular at
$(\xb,0)$, then it follows from Theorem \ref{ThMixedRegSubreg}
that there exist $0\not=\lambda_0\in N(0;\R_-)$ and
$\lambda=(\lambda_1,\lambda_2)\in
N(F(\xb);\Omega)=N(F_1(\xb);\Omega_1)\times N(F_2(\xb);\Omega_2)$
with $\lambda_0\nabla f(\xb)^T+\nabla F(\xb)^T\lambda=0$. Hence
$\lambda_0>0$ and since $N(F(\xb),\Omega)$ is a cone, it follows
that
$\frac{\lambda}{\lambda_0}=(\frac{\lambda_1}{\lambda_0},\frac{\lambda_2}{\lambda_0})\in
\Lambda^1(\xb)$. If $M=(M_1,M_2)$ is not metrically subregular at
$(\xb,0)$, then it is also not mixed regular/subregular at
$(\xb,0)$ and  by Theorem \ref{ThMixedRegSubreg} we obtain
$\Lambda^0(\xb)\not=\emptyset$. Of course, the second-order
condition \eqref{EqSecOrdNec} is trivially fulfilled for $u=0$.

Now let $u\not=0$. If $M$ is metrically subregular in direction
$u$  at $(\xb,0)$, we choose $\eta$ such that $({\cal
M}_1^{\eta,\xb},M)$ is not metrically subregular in direction $u$
according to Proposition \ref{PropOptCond} and apply Theorem
\ref{ThDirSubReg}. Therefore there exists $0\not=\lambda_0\in
N(0;\R_-;\nabla f(\xb)u)$ and $\lambda=(\lambda_1,\lambda_2)\in
N(F(\xb);\Omega;\nabla F(\xb)u)=N(F_1(\xb);\Omega_1; \nabla
F_1(\xb)u)\times N(F_2(\xb);\Omega_2;\nabla F_2(\xb)u)$ with
$\lambda_0\nabla f(\xb)^T+\nabla F(\xb)^T\lambda=0$. In addition,
if $F$ and $f$ are twice Fr\'echet differentiable, we can assume
that \[u^T\nabla_x^2(\lambda_0f+\lambda^T
F)(\xb)u=u^T\nabla_x^2{\cal L}(\xb,\lambda_0,\lambda)u\geq 0.\]
Then $\lambda_0>0$,
$\frac{\lambda}{\lambda_0}=(\frac{\lambda_1}{\lambda_0},\frac{\lambda_2}{\lambda_0})\in
\Lambda^1(\xb;u)$ and $u^T\nabla_x^2{\cal
L}(\xb,1,\frac{\lambda}{\lambda_0})u\geq 0$.

If $M$ is not metrically subregular in direction $u$, we can
apply Theorem \ref{ThDirSubReg} to $(M_1,M_2)$ to conclude that
the assertions hold with
$(\lambda_1,\lambda_2)\in\Lambda^0(\xb;u)$.
\end{proof}

In case of the nonlinear programming problem, where the
constraints are given by smooth inequality and equality
constraints, i.e. $\Omega$ is a convex polyhedron of the form
$\R_-^l\times\{0\}^p$, by Lemma \ref{LemRelationsMult} we have
$\hat\Lambda^{\lambda_0}(\xb)=\hat\Lambda^{\lambda_0}(\xb;u)=\Lambda^{\lambda_0}(\xb;u)=
\Lambda^{\lambda_0}(\xb)$ for every $\lambda_0\geq0$ and every
critical direction $u$. Moreover, in case that $m_2=0$, i.e. we do
not identify some part of the constraints being subregular, the
sets $\Lambda^1(\xb)$ and $\{(\lambda_0,\lambda)\mv \lambda_0\geq
0,\lambda\in\Lambda^{\lambda_0}(\xb)\}$ coincide with the sets of
multipliers fulfilling the Karush-Kuhn-Tucker conditions and the
Fritz John conditions, respectively. Hence, in this case we can
recover from Theorem \ref{ThBaseOptCond} the usual first-order and
second-order necessary conditions of both Fritz-John and
Karush-Kuhn-Tucker type,  see e.g. \cite{Ben80,Iof79c,LeMiOs78}. However, the statement of
Theorem \ref{ThBaseOptCond} is a little bit stronger. In Example
\ref{ExGuignard} below we will give an example of a nonlinear
programming problem, where we cannot reject a non-optimal point by
the usual necessary optimality conditions, but by using our theory
based on directional metric subregularity we can. However, we will
not work out the theory for the nonlinear programming problem, but
for the more general problem with disjunctive constraints.

Next we relate the first-order optimality conditions contained in
Theorem \ref{ThBaseOptCond} with some  stationarity concepts:
\begin{definition}
Let $\xb$ be feasible for the problem \eqref{EqMPDC}. We say  that
\begin{enumerate}
\item $\xb$ is {\em B-stationary} (Bouligand-stationary) if
\[\nabla f(\xb)u\geq 0\ \forall u\in T(\xb;{\cal F}),\]
i.e. $-\nabla f(\xb)^T\in \hat N(\xb;{\cal F})$,
 \item $\xb$ is {S-stationary} (strongly stationary) if
\[\hat \Lambda^1(\xb)\not=\emptyset,\]
\item $\xb$ is {\em M-stationary} (Mordukhovich-stationary) if
\[\Lambda^1(\xb)\not=\emptyset,\]
\item $\xb$ is {\em extended M-stationary} if
\[\Lambda^1(\xb;u)\not=\emptyset\ \forall u\in{\cal C}(\xb).\]
\end{enumerate}
\end{definition}
It is well known that a local minimizer is  B-stationary. The
B-stationarity condition expresses that at a local minimizer there
does not exist any feasible descent direction.

Our definition of B-stationarity corresponds  to the definition of
B-stationarity for MPECs as can be found in the monograph
\cite{LuPaRa96}. The definitions of M-stationarity and
S-stationarity were introduced in \cite{FleKanOut07} and are in
accordance with the definitions for MPECs \cite{SchSch00}. The
definition of extended M-stationarity is motivated by Theorem
\ref{ThBaseOptCond}.

\begin{lemma}\label{LemS_StatB_Stat}If $\xb$ is S-stationary
 then $\nabla f(\xb)u\geq 0$ $\forall u\in T_{\rm lin}(\xb)$ and consequently $\xb$ is B-stationary.
\end{lemma}
\begin{proof}Consider an arbitrarily fixed  direction  $u\in T_{\rm lin}(\xb)$.
Since $\nabla F(\xb)u\in T(F(\xb);\Omega)$,  for every
$\lambda\in\hat\Lambda^1(\xb)\subset \hat N(0;T(F(\xb);\Omega))$
we have $-\nabla f(\xb)u=\lambda^T\nabla F(\xb)u\leq 0$.
\end{proof}

Hence, every S-stationary solution is also B-stationary. However,
the converse direction is only true under some relatively strong
constraint qualification.

\begin{definition}\label{DefLICQ_u}
Let $u\in T_{\rm lin}(\xb)$. We say that the {\em linear
independence constraint qualification condition in direction $u$}
(LICQ(u)) holds at $\xb$ if there is some subspace $L\subset \R^m$
such that
\[T(\nabla F(\xb)u;T(F(\xb);\Omega))+L\subset T(\nabla F(\xb)u;T(F(\xb);\Omega))\]
and
\[\nabla F(\xb)\R^n+L=\R^m.\]
\end{definition}
Note that LICQ(0) is related to the non-degeneracy condition
\cite[(4.172)]{BonSh00}. The notation LICQ(u) is motivated by the
fact that for the MPEC \eqref{EqMPEC} the condition LICQ(0) is
equivalent with the well-known MPEC-LICQ constraint qualification,
as we will see in Section \ref{SecMPEC}. In particular, for the
nonlinear programming problem LICQ(0) is the same as LICQ, i.e.,
the gradients of the active constraints are linearly independent.

\begin{lemma}\label{LemB_StatS_Stat}If $\xb$ is B-stationary and
LICQ(0) holds at $\xb$, then $\xb$ is also S-stationary.
\end{lemma}
\begin{proof}
Follows from Proposition \ref{PropLICQ} below.
\end{proof}

For MPECs \eqref{EqMPEC} it is well known, that the  weaker
constraint qualification GMFCQ \eqref{EqMCCQ} does not guarantee
S-stationarity of a local minimizer. Although GMFCQ implies that
$T(\xb;{\cal F})=T_{\rm lin}(\xb)$, we only have the inclusion
$\nabla F(\xb)^T\hat N(F(\xb);\Omega)\subset \hat N(0;T_{\rm
lin}(\xb))$ (see \cite[Theorem 6.14]{RoWe98}) and equality, which
would be required for S-stationarity, is not fulfilled in general.

From Theorem \ref{ThBaseOptCond} it follows  that a local
minimizer is M-stationary if there exists one critical direction
$u$ such that the multifunction $M$ associated with the
constraints is metrically subregular in direction $u$. Further, if
$M$ is metrically subregular in every critical direction $u$, then
a local minimizer is also extended M-stationary. Note that the
requirement that $M$ is metrically subregular in one respectively
any critical direction is  not a constraint qualification in
general, since the objective function is also involved in the
definition of critical directions. The only exception is the
trivial critical direction $u=0$, because metric subregularity of
$M$ in direction $0$ means metric subregularity of $M$. Hence,
under the constraint qualification  of metric subregularity of the
constraint mapping $M$ we have that every local minimizer is also
extended M-stationary.

We will now show that this holds true under some weaker constraint
qualification than metric subregularity. Actually we will prove
that under a suitable weak constraint qualification extended M-
stationarity is equivalent to B-stationarity.

\begin{definition}(cf. \cite{FleKanOut07})
We say that the {\em generalized} (or dual) {\em Guignard constraint qualification}
(GGCQ) holds at the feasible point $\xb\in{\cal F}$ if
\[\hat N(\xb;{\cal F})=\hat N(0;T_{\rm lin}(\xb)).\]
\end{definition}

Recall that a polyhedral cone is finitely  generated \cite[\S
19]{Ro70}. For each $i=1,\ldots,\bar p$ the set $P_i$ is
polyhedral and therefore both the tangent cone $T(F(\xb);P_i)$ and
the cone $L_i:=\{u\in\R^n\mv\nabla F(\xb)u\in T(F(\xb);P_i)\}$ are
polyhedral cones and consequently finitely generated. Hence,
$\co(\bigcup_{i=1}^{\bar p} L_i) = \co\{u\in\R^n\mv \nabla
F(\xb)u\in T(F(\xb);\Omega)\}$ is also finitely generated, at
least by the union of the generators for $L_i$, but maybe by a
smaller set. That is, there exists a set ${\cal
U}=\{u_1,\ldots,u_N\}\subset T_{\rm lin}(\xb)$ such that
\begin{equation}\label{EqFiniteGen}
\co T_{\rm lin}(\xb)=\{\sum_{i=1}^N \alpha_iu_i\mv \alpha_i\geq 0, i=1,\ldots,N\}
\end{equation}

\begin{theorem}
\label{ThEquivM-B-stat}Assume that GGCQ  is satisfied  at the point $\xb\in{\cal F}$ feasible for the problem \eqref{EqMPDC} and let $\co T_{\rm lin}(\xb)$ be finitely generated by the set ${\cal U}=\{u_1,\ldots,u_N\}\subset T_{\rm lin}(\xb)$. Then the following statements are equivalent:
\begin{enumerate}
\item[(a)] $\xb$ is B-stationary.
\item[(b)] $\nabla f(\xb) u\geq0$ $\forall u\in T_{\rm lin}(\xb)$.
\item[(c)] $\xb$ is extended M-stationary.
\item[(d)] For every direction $u\in {\cal U}\cap {\cal C}(\xb)$ there holds $\Lambda^1(\xb;u)\not=\emptyset$.
\end{enumerate}
\end{theorem}

\begin{proof}Using the equivalences
\[\mbox{$\xb$ is B-stationary}\Leftrightarrow -\nabla f(\xb)\in \hat N(\xb;{\cal F}),\quad -\nabla f(\xb)\in\hat N(0;T_{\rm lin}(\xb))\Leftrightarrow \nabla f(\xb) u\geq0\ \forall u\in T_{\rm lin}(\xb)\]
we obtain  (a)$\Leftrightarrow$(b) from GGCQ.

Next we show (b)$\Rightarrow$(c). Statement (b) means that $u=0$ is a solution of the problem
\[\min \nabla f(\xb)u \quad \text{subject to}\quad \nabla F(\xb)u\in T(F(\xb);\Omega).\]
Since $\Omega$ is the  union of finitely many polyhedra, there is
a neighborhood $U$ of $F(\xb)$ such that $\Omega\cap
U=(F(\xb)+T(F(\xb);\Omega))\cap U$ and thus $u=0$ is a local
minimizer of the problem
\begin{equation}
\label{EqLinProbl}\min \nabla f(\xb)u \quad \text{subject to}\quad F(\xb)+\nabla F(\xb)u\in \Omega
\end{equation}
The constraint mapping  $u\rightrightarrows F(\xb)+\nabla F(\xb)u-
\Omega$ is a polyhedral multifunction and therefore metrically
subregular at $0$ by Robinson's result \cite{RoWe98}. Hence we can
apply Theorem \ref{ThBaseOptCond} to obtain that $0$ is extended
M-stationary for the problem \eqref{EqLinProbl}. But it is easy to
see that extended M-stationarity of $u=0$ for the problem
\eqref{EqLinProbl}  is equivalent to  extended M-stationarity of
$\xb$ for the problem \eqref{EqMPDC} and the assertion follows.

The implication (c)$\Rightarrow$(d) is  obviously true. Finally we
show (d)$\Rightarrow$(b). Since $\Lambda^1(\xb;u)\not=\emptyset$
implies $\nabla f(\xb)u=0$ by \eqref{EqLambda0Nabla_f}, we see
that (d) implies $\nabla f(\xb)u\geq 0$ $\forall u\in{\cal U}$.
Since $\co T_{\rm lin}(\xb)$ is generated by $\cal U$ we obtain
$\nabla f(\xb)u\geq 0$ $\forall u\in\co T_{\rm lin}(\xb)$ and (b)
follows.
\end{proof}

\begin{remark}
Assumption GGCQ is only needed to prove $(a)\Leftrightarrow (b)$.
Since we always have $T(\xb;{\cal F})\subset T_{\rm lin}(\xb)$ and
consequently
\begin{equation}\label{EqInclGGCQ}\hat N(\xb;{\cal F})\supset \hat N(0;T_{\rm lin}(\xb)),\end{equation}
we obtain that the relations
\[(a)\Leftarrow (b)\Leftrightarrow (c)\Leftrightarrow (d)\]
are valid without assuming GGCQ.
\end{remark}

Obviously extended M-stationarity implies  M-stationarity. Putting
all together we obtain the following picture:
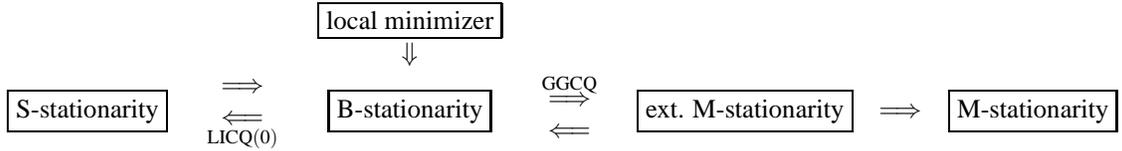
\[\begin{array}{ccccccc}&&\fbox{local minimizer}\\&&\Downarrow\\
\fbox{S-stationarity}&
\begin{array}{c}\Longrightarrow\\\mathop{\Longleftarrow}\limits_{{\rm LICQ(0)}}\end{array}
&\fbox{B-stationarity}&\begin{array}{c}\mathop{\Longrightarrow}\limits^{\rm
GGCQ}\\\Longleftarrow\end{array}&\fbox{ext.
M-stationarity}&\Longrightarrow&\fbox{M-stationarity}
\end{array}
\]

From Theorem \ref{ThEquivM-B-stat}  we derive that GGCQ is a
constraint qualification for a local minimizer to be extended
M-stationary. The following proposition states, that GGCQ is in
some sense the weakest possible constraint qualification ensuring
extended M-stationarity.
\begin{proposition}
Assume that $\xb\in{\cal F}$ is an extended M-stationary point of
\[\min f(x)\mbox{ subject to } F(x)\in\Omega\]
for every continuously differentiable function $f:\R^n\to\R$ with
$\xb$ being a local minimizer. Then GGCQ is fulfilled at $\xb$.
\end{proposition}
\begin{proof}
By contraposition. Assume that GGCQ is not fulfilled at $\xb$.
Then, by taking into account \eqref{EqInclGGCQ}, there is some
$\xi\in\hat N(\xb;{\cal F})\setminus \hat N(0;T_{\rm lin}(\xb))$
and thus $\xi^T u>0$ for some $u\in T_{\rm lin}(\xb)$. By
\cite[Theorem 1.30]{Mo06a} there is some continuously differentiable function
$\varphi$ with $\nabla \varphi(\xb)=\xi^T$ such that $\varphi$ attains
its global maximum over ${\cal F}$ at $\xb$. Therefore, by
taking $f=-\varphi$, $u$ is a critical direction fulfilling
the extended M-stationarity condition $\Lambda^1(\xb;u)\not=\emptyset$ and by \eqref{EqLambda0Nabla_f} we obtain  $\nabla
f(\xb)u=-\xi^T u=0$, a contradiction.
\end{proof}

 GGCQ is
very difficult to verify in general. Hence we present another
constraint qualification  stronger than GGCQ but verifiable:
\begin{definition}
We say that the {\em weak directional metric subregularity constraint qualification} (WDMSCQ) is satisfied at the point $\xb$ feasible for \eqref{EqMPDC}, if there is a finite set ${\cal U}\subset T_{\rm lin}(\xb)$ generating $\co T_{\rm lin}(\xb)$ such that $M(x)=F(x)-\Omega$ is metrically subregular in every direction $u\in {\cal U}$ at $(\xb,0)$.
\end{definition}

\begin{proposition} WDMSCQ $\Rightarrow$ GGCQ.
\end{proposition}
\begin{proof}
By contraposition. Assuming that GGCQ is not fulfilled at $\xb$,
by taking into account \eqref{EqInclGGCQ}, there is some
$\xi\in\hat N(\xb;{\cal F})\setminus \hat N(0;T_{\rm lin}(\xb))$ and thus $\xi^Tu>0$ for some $u\in T_{\rm lin}(\xb)$. Since $u$ can be
represented as a nonnegative linear combination of
$u_1,\ldots,u_N$, there exists $\tilde u\in {\cal U}$ with
$\xi^T\tilde u>0$. Because $\Omega$ is the union of finitely many
polyhedral sets, there is some neighborhood $U$ of $F(\xb)$ such
that $\Omega\cap U=(F(\xb)+T(F(\xb);\Omega))\cap U$ and therefore
$F(\xb)+t\nabla F(\xb)\tilde u\in\Omega$ for all $t\geq 0$
sufficiently small. Since $M$ is assumed to be metrically
subregular in direction $\tilde u$ there is some $\kappa>0$ such
that
\[\dist{\xb+t\tilde u,{\cal F}}=\dist{\xb+t\tilde u, M^{-1}(0)}\leq\kappa\dist{0,M(\xb+t\tilde u)}\leq \kappa \norm{F(\xb)+t\nabla F(\xb)\tilde u-F(\xb+t\tilde u)}\]
holds for all $t\geq 0$ sufficiently small. This implies that for every $t>0$ we can find some $x_t\in{\cal F}$ satisfying
\[0\leq\limsup_{t\downarrow 0}\norm{\frac {x_t-\xb}t-\tilde u}=\limsup_{t\downarrow 0}\norm{\frac {x_t-(\xb+t\tilde u)}t}\leq \limsup_{t\downarrow0}\frac{\kappa\norm{F(\xb)+t\nabla F(\xb)\tilde u-F(\xb+t\tilde u)}}t=0.\]
Hence $\tilde u\in T(\xb,{\cal F})$ and  because of $\xi\in \hat N(\xb;{\cal F})$  we have  $\xi^T\tilde u\leq 0$ contradicting $\xi^T\tilde u>0$.
\end{proof}

Note that Theorem \ref{ThDirSubReg}  provides point based
conditions to verify WDMSCQ. We reformulate these conditions in
the following lemma:

\begin{lemma}\label{LemSuffDirSubreg}Let $\xb$ be feasible for the problem \eqref{EqMPDC} and let $u\in T_{\rm lin}(\xb)$.
If either
\begin{enumerate}
\item $\Lambda^0(\xb;u)=\emptyset$, or
\item $F$ is twice Fr\'echet differentiable at $\xb$ and $u^T\nabla_x^2{\cal L}(\xb,0,\lambda)u<0$ $\forall \lambda\in \Lambda^0(\xb;u)$,
\end{enumerate}
then $M$ is metrically subregular in direction $u$.
\end{lemma}

This lemma states that, if for a critical direction $u$ either the
first-order necessary optimality condition or the second-order
necessary optimality conditions cannot be fulfilled with
multiplier $\lambda_0=0$, then the constraint mapping $M$ is
metrically subregular in direction $u$.

\begin{example}\label{ExGuignard}
Consider the nonlinear programming problem
\begin{eqnarray*}
\min &&-x_1\\
 -x_1+\vert x_1\vert^{\frac 32}&\leq& 0,\\ -x_2&\leq& 0,\\
 (x_1-2x_2)(x_2-2x_1)&\leq& 0,\\
\end{eqnarray*}
at $\xb=(0,0)$. Then $\xb$ is not a local minimizer and we will
demonstrate how this  can be verified by our necessary conditions.
Note that we cannot reject $\xb$  as a local minimizer by the
usual second-order necessary conditions of nonlinear programming,
because the term $\vert x_1\vert^{\frac 32}$  is not twice
Fr\'echet differentiable at $x_1=0$.

Consider the multifunction  $\tilde M(x):=\tilde
F(x)-\tilde\Omega:=(x_1-2x_2)(x_2-2x_1)-\R_-$ and let
$u=(u_1,u_2)\in T_{\rm lin}(\xb)=\R_+^2$ with
$(u_1-2u_2)(u_2-2u_1)< 0$. We shall now show by using Lemma
\ref{LemSuffDirSubreg} that $\tilde M$ is metrically subregular in
direction $u$ at $(\xb,0)$. Straightforward calculations yield
that the corresponding set of multipliers is
$\tilde\Lambda^0(\xb;u)=\R_+\setminus\{0\}$ and for every
$\lambda>0$ we have $u^T\nabla_x^2\tilde{\cal
L}(\xb,0,\lambda)u=2\lambda(u_1-2u_2)(u_2-2u_1)< 0$ establishing
directional metric subregularity.

Hence, for $u=(u_1,u_2)\in T_{\rm lin}(\xb)$ we can use the splitting of the constraint mapping
\begin{eqnarray*}
M_1(x)&=&\left(\begin{array}{c}-x_1+\vert x_1\vert^{\frac 32}\\-x_2\end{array}\right)-\R_-^2,\ M_2(x)=(x_1-2x_2)(x_2-2x_1)-\R_-\mbox{ if $(u_1-2u_2)(u_2-2u_1)< 0$,}\\
M_1(x)&=&\left(\begin{array}{c}-x_1+\vert x_1\vert^{\frac 32}\\-x_2\\(x_1-2x_2)(x_2-2x_1)\end{array}\right)-\R_-^3,\ M_2(x)=\{0\}\mbox{ if $(u_1-2u_2)(u_2-2u_1)\geq 0$.}
\end{eqnarray*}

Now we consider the critical direction $u=(1,0)\in{\cal C}(\xb)=T_{\rm lin}(\xb)=\R_+^2$. The Lagrange function is given by
\[{\cal L}(x,\lambda_0,\lambda_1,\lambda_2,\lambda_3)=-\lambda_0x_1+\lambda_1(-x_1+\vert x_1\vert^{\frac 32})-\lambda_2x_2+\lambda_3(x_1-2x_2)(x_2-2x_1)\]
and for $\lambda_0\geq 0$ the set $\Lambda^{\lambda_0}(\xb;u)$ is given by
\[\Lambda^{\lambda_0}(\xb;(1,0))=\left\{(\lambda_1,\lambda_2,\lambda_3)\mv\begin{array}{l}\nabla _x{\cal L}(\xb,\lambda_0,\lambda_1,\lambda_2,\lambda_3)=(-\lambda_0-\lambda_1,-\lambda_2)=0\\
(\lambda_1,\lambda_2,\lambda_3)\in N((-1,0,0);\R_-^3)=\{0\}\times \R^2_+\\
\lambda_0+\vert \lambda_1\vert +\vert \lambda_2\vert>0\end{array}\right\}=\emptyset.\]
Hence the first-order optimality conditions of Theorem \ref{ThBaseOptCond} are violated and $\xb$ is not a local minimizer.

Obviously $T_{\rm lin}(\xb)=\R^2_+$ is generated by the two
directions $u=(1,0)^T$ and $v=(0,1)^T$. We have just established
$\Lambda^0(\xb;u)=\emptyset$ showing metric subregularity of the
constraint mapping in direction $u$ by Lemma
\ref{LemSuffDirSubreg}. In the same way one can also show metric
subregularity in direction $v$ and thus WDMSCQ and consequently
GGCQ is fulfilled. Note that $T(\xb;{\cal F})=\{(x_1,x_2)\mv 0\leq
x_1\leq x_2/2\}\cup\{(x_1,x_2)\mv 0\leq x_2\leq
x_1/2\}\}\not=T_{\rm lin}(\xb)$, i.e. the so-called {\em Abadie
constraint qualification} fails to hold but nevertheless we were
able to prove GGCQ.

Further, the mapping $M(x)=F(x)-\Omega$ is not metrically
subregular. E.g., consider the points $x_t:=(t, t)$ for arbitrary
$t>0$ satisfying $\lim_{t\downarrow 0}(x_t-\xb)/t=(1,1)$,
$F(x_t)=(-t+t^{\frac 32},-t,t^2)\not \in\Omega$,
$\dist{0,M(x_t)}=t^2$ and $\dist{x_t,M^{-1}(0)}=t/\sqrt{5}$ for
$0<t<1$, showing that $M$ is not metrically subregular in
direction $(1,1)$. Similar arguments show that metric
subregularity also fail to hold in every direction $u$ with
$(u_1-2u_2)(u_2-2u_1)\geq 0$. Note that the lack of metric
subregularity would also follow from the lack of the Abadie
constraint qualification.
\end{example}

We consider now second-order sufficient conditions. Consider the
following definition owing to Penot \cite{Pen98}:
\begin{definition}\label{DepEssLocMin}
We say that $\xb\in\R^n$ is an {\em essential local minimizer of
second order} for problem \eqref{EqMPDC}, if $\xb$ is feasible and
there exists some neighborhood $U$ of $\xb$ and some real
$\beta>0$ such that
\[\max\{f(x)- f(\xb), \dist{F(x),\Omega}\}\geq \beta\norm{x-\xb}^2\ \forall x\in U.\]
\end{definition}
Obviously at an essential local minimizer of second order the following {\em quadratic growth condition} is fulfilled:
\[f(x)\geq f(\xb)+\beta\norm{x-\xb}^2\ \forall x\in{\cal F}\cap U.\]
This quadratic growth  condition is also sufficient for $\xb$ to
be an essential local minimizer of second order, if the constraint
mapping $M$ is metrically subregular at $(\xb,0)$ and $f$ is
Lipschitz near $\xb$. To see this one could use similar arguments
as in \cite[Section 3]{Gfr06} by noting that convexity of $\Omega$
is not needed and the assumption of metric regularity used in
\cite{Gfr06} can be replaced by assuming metric subregularity.

\begin{theorem}
Assume that $\xb$ is a local minimizer but not an essential local
minimizer of second order for the problem \eqref{EqMPDC}. Then
there exists a twice continuously differentiable function
$h=(\delta f,\delta F):\R^n\to\R\times\R^m$ with $h(\xb)=0$,
$\nabla h(\xb)=0$, $\nabla^2 h(\xb)=0$ such that $\xb$ is not a
local minimizer for the problem
\[\min (f+\delta f)(\xb)\mbox{ subject to }(F+\delta F)(x)\in\Omega.\]
\end{theorem}
\begin{proof}
Follows from  the proof of \cite[Theorem 3.5]{Gfr06} by recognizing that
convexity of $\Omega$ is not needed in that proof.
\end{proof}

From this statement it follows that a characterization of $\xb$
being an essential local minimizer of second order is the weakest
possible sufficient second-order optimality condition which uses
solely function values and derivatives up to order 2 at the point
$\xb$.

For each $u\in T_{\rm lin}(\xb)$ we now denote by ${\cal P}(u)$ the index set
\[{\cal P}(u):=\{i\in{1,\ldots,\bar p}\mv \nabla F(\xb)u\in T(F(\xb); P_i)\}\]
Since $T(F(\xb);\Omega)=\bigcup_{i=1}^{\bar p}T(F(\xb);P_i)$ we have ${\cal P}(u)\not=\emptyset$ for every $u\in T_{\rm lin}(\xb)$.
\begin{lemma}\label{LemCharEssLocMin}
Let $\xb$ be feasible for the problem \eqref{EqMPDC} and let $f$, $F$ be twice Fr\'echet differentiable at $\xb$. Then the following statements are equivalent:
\begin{enumerate}
\item[(a)] $\xb$ is an essential local minimizer of second order.
\item[(b)] For every nonzero critical direction $0\not=u\in{\cal
C}(\xb)$ and every $i\in{\cal P}(u)$ there  exists some multiplier
$(\lambda_0,\lambda)\in\hat N(\nabla f(\xb)u;\R_-)\times\hat
N(\nabla F(\xb)u;T(F(\xb);P_i))$ with $\nabla_x {\cal
L}(\xb,\lambda_0,\lambda)=0$ and
\begin{equation}
u^T\nabla_x^2 {\cal L}(\xb,\lambda_0,\lambda)u>0.
\end{equation}
\item[(c)] For every nonzero critical direction $0\not=u\in{\cal C}(\xb)$ there does not exist $v\in\R^n$ with
\begin{eqnarray}\label{EqEssMinCond1}
\nabla f(\xb)v+\frac 12 u^T\nabla^2f(\xb)u &\in& T(\nabla f(\xb)u;\R_-)\\
\label{EqEssMinCond2} \nabla F(\xb)v+\frac 12 u^T\nabla^2F(\xb)u&\in& T(\nabla F(\xb)u; T(F(\xb);\Omega))
\end{eqnarray}
\end{enumerate}
\end{lemma}
\begin{proof}Let $\bar {\cal P}:=\{i\in\{1,\ldots,\bar p\}\mv F(\xb)\in P_i\}$. Then there is a neighborhood $U$ of $\xb$ such that $\dist{F(x),\Omega}=\min_{i\in\bar{\cal P}}\dist{F(x),P_i}$ $\forall x\in U$ and therefore $\xb$ is an essential local minimizer of second order if and only if for each $i\in\bar{\cal P}$ the point $\xb$ is an essential local minimizer of second order for the problem
\[\min f(x)\mbox{ subject to } F(x)\in P_i\]
Hence, by using \cite[Theorems 5.4,5.11]{Gfr06} we obtain that $\xb$ is an essential local minimizer of second order if and only if for each $i\in\bar{\cal P}$ and every $u$ with $\nabla f(\xb)u\leq 0$ and $\nabla F(\xb)u\in T(F(\xb);P_i)$ there is some multiplier $(\lambda_0,\lambda)\in\R_+\times\hat N(F(\xb);P_i)=\hat N(0;\R_-)\times \hat N(0;T(F(\xb);P_i))$ with $\nabla_x {\cal L}(\xb,\lambda_0,\lambda)=0$ and
$u^T\nabla_x^2 {\cal L}(\xb,\lambda_0,\lambda)u>0$. Hence $\lambda_0\nabla f(\xb)u\leq 0$ and $\lambda^T \nabla F(\xb)u\leq 0$ and from $\nabla_x {\cal L}(\xb,\lambda_0,\lambda)u=0$ we conclude $\lambda_0\nabla f(\xb)u=-\lambda^T\nabla F(\xb)u=0$. Thus $(\lambda_0,\lambda)\in\hat N(\nabla f(\xb)u;\R_-)\times\hat N(\nabla F(\xb)u;T(F(\xb);P_i))$ and this establishes the equivalence (a)$\Leftrightarrow$(b).

Next we  show the equivalence (b)$\Leftrightarrow(c)$: Let
$0\not=u\in{\cal C}(\xb)$ and $i\in{\cal P}(u)$ be fixed and define
$A:=-(\nabla f(x)^T\ \vdots\ \nabla F(\xb)^T)$, $L:=\hat N(\nabla
f(\xb)u;\R_-)\times\hat N(\nabla F(\xb)u;T(F(\xb);P_i))$,
$K:=\{\xi\in L\mv A\xi=0\}$ and $b:=(\frac 12 u^T\nabla^2
f(\xb)u,(\frac 12u^T\nabla F(\xb)u)^T)^T$. We have  $b\not\in
K^\circ$ if and only if there is some $\xi=(\lambda_0,\lambda)\in
L$ satisfying $A\xi=-\nabla_x{\cal L}(\xb,\lambda_0,\lambda)=0$
such that $\xi^Tb=\frac 12 u^T\nabla_x^2{\cal
L}(\xb,\lambda_0,\lambda)u>0$. Since $L$ is a
polyhedral cone, by the generalized Farkas Lemma \cite[Proposition 2.201]{BonSh00} we have
$K^\circ= \Range A^T + L^\circ$ and, together with $L^\circ=T(\nabla
f(\xb)u;\R_-)\times T(\nabla F(\xb)u; T(F(\xb);P_i))$, it follows that
statement (b) is equivalent to the condition that for every $0\not=u\in{\cal
C}(\xb)$ and each $i\in{\cal P}(u)$ the system
\begin{eqnarray*}
\nabla f(\xb)v+\frac 12 u^T\nabla^2f(\xb)u &\in& T(\nabla f(\xb)u;\R_-)\\
\nabla F(\xb)v+\frac 12 u^T\nabla^2F(\xb)u&\in& T(\nabla F(\xb)u; T(F(\xb);P_i))
\end{eqnarray*}
does not have a solution $v$. Noting that $T(\nabla F(\xb)u;
T(F(\xb);\Omega))=\bigcup_{i\in{\cal P}(u)}T(\nabla F(\xb)u;
T(F(\xb);P_i))$ we conclude (b)$\Leftrightarrow$(c).
\end{proof}

\begin{theorem}\label{ThEssLocMin}
Let $\xb$ be feasible for the problem  \eqref{EqMPDC} and let $f$, $F$ be twice Fr\'echet differentiable at $\xb$. If $\xb$ is an essential local minimizer of second order then for every critical direction $0\not=u\in{\cal C}(\xb)$ there is some pair $(\lambda_0,\lambda)\in \R_+\times\R^m$ with $\lambda\in\Lambda^{\lambda_0}(\xb;u)$ such that
\begin{equation}\label{EqPosDefLagr}u^T\nabla_x^2{\cal L}(\xb,\lambda_0,\lambda)u>0.
\end{equation}
Conversely, if for every critical direction $0\not=u \in{\cal C}(\xb)$ there is some pair $(\lambda_0,\lambda)\in \R_+\times\R^m$ fulfilling $\lambda\in\hat \Lambda^{\lambda_0}(\xb;u)$ and \eqref{EqPosDefLagr}, then $\xb$ is an essential local minimizer of second order.
\end{theorem}
\begin{proof}
Firstly assume that $\xb$ is an essential local minimizer of
second order and consider the problem
\begin{eqnarray*}\min&& f(x)-\beta\norm{x-\xb}^2\\
\mbox{subject to }&& F(x)\in\Omega,
\end{eqnarray*}
where $\beta>0$ is chosen according to the Definition
\ref{DepEssLocMin}. Since $\xb$ is a local minimizer of the above
problem, by Theorem \ref{ThBaseOptCond}, we can easily get the
first part of this theorem.

To show the second assertion we use the equivalence (a)$\Leftrightarrow$(b) of Lemma \ref{LemCharEssLocMin}. Let $0\not=u\in{\cal C}(\xb)$ be arbitrarily fixed and choose $\lambda_0\geq 0$ and $\lambda\in\hat\Lambda^{\lambda_0}(\xb;u)$ with $u^T\nabla_x^2{\cal L}(\xb,\lambda_0,\lambda)u>0$. By the definition of $\hat\Lambda^{\lambda_0}(\xb;u)$ we have $\nabla_x{\cal L}(\xb,\lambda_0,\lambda)=0$ and we will now show that $(\lambda_0,\lambda)\in\hat N(\nabla f(\xb)u;\R_-)\times\hat N(\nabla F(\xb)u;T(F(\xb);P_i))$ for each $i\in{\cal P}(u)$. Because of $\hat\Lambda^{\lambda_0}(\xb;u)\subset \Lambda^{\lambda_0}(\xb;u)$ and \eqref{EqLambda0Nabla_f} we have $\lambda_0\in \hat N(\nabla f(\xb)u;\R_-)$. Further
\begin{eqnarray*}\lambda&\in&\hat N(\nabla F(\xb)u;T(F(\xb);\Omega))=\hat N(0;T(\nabla F(\xb)u;T(F(\xb);\Omega)))\\
&=&\hat N(0;\bigcup_{i\in{\cal P}(u)}T(\nabla F(\xb)u;T(F(\xb);P_i)))
=\bigcap_{i\in{\cal P}(u)}\hat N(0;T(\nabla F(\xb)u;T(F(\xb);P_i)))\\
&=&\bigcap_{i\in{\cal P}(u)}\hat N(\nabla F(\xb)u;T(F(\xb);P_i))\end{eqnarray*}
and thus our assertion is proved.
\end{proof}

In case of $\bar p=1$, i.e. $\Omega$ is a convex polyhedron, we
have $\hat \Lambda^{\lambda_0}(\xb;u)=\Lambda^{\lambda_0}(\xb;u)$
$\forall \lambda_0\geq 0$, $\forall u\in{\cal C}(\xb)$ by Lemma
\ref{LemRelationsMult} and thus \eqref{EqPosDefLagr} is an
equivalent condition for $\xb$ being an essential local minimizer
of second order. In particular, this is the case for the nonlinear
programming problem, where \eqref{EqPosDefLagr} is nothing else
than the second-order sufficient condition of nonlinear
programming \cite{Ben80, Iof79c}.

We will now show that also under LICQ(u) the sets
$\hat\Lambda^{\lambda_0}(\xb;u)$ and $\Lambda^{\lambda_0}(\xb;u)$
coincide.
\begin{lemma}\label{LemBasLICQ}
Assume that LICQ(u) holds for $u\in T_{\rm lin}(\xb)$. Then
$\Lambda^0(\xb;u)=\emptyset$.
\end{lemma}
\begin{proof}
Assume that there is some $\lambda\in\Lambda^0(\xb;u)$. Then $\lambda\not=0$ and  there is some $v\in\R^n$ and $w\in L$ such that $\nabla F(\xb)v+ w=\lambda$, implying
\[0<\lambda^T\lambda=\lambda^T\nabla F(\xb)v+\lambda^Tw=\lambda^T w\]
because of $\nabla_x{\cal L}(\xb;0,\lambda)=\lambda^T\nabla
F(\xb)=0$. By Lemma \ref{LemBasicPropCone} there is some $z\in
T(\nabla F(\xb)u;T(F(\xb);\Omega)))$ with $\lambda\in\hat
N(z;T(\nabla F(\xb)u;T(F(\xb);\Omega)))$ and therefore, since
$z+L\subset T(\nabla F(\xb)u;T(F(\xb);\Omega))$ and $\alpha w\in L$ $\forall\alpha\in\R$, we obtain $\lambda^T(z+\alpha w)=\lambda^Tz +\alpha\lambda^Tw\leq0$ $\forall \alpha\in\R$ implying the contradiction $\lambda^Tw=0$.
\end{proof}
\begin{proposition}\label{PropLICQ}
Assume that $\xb$ is B-stationary for the problem \eqref{EqMPDC}. Then for every critical direction $u\in{\cal C}(\xb)$ fulfilling  LICQ(u) there is a unique element $\lambda_u\in\R^m$ such that
\[\Lambda^1(\xb;u)=\hat \Lambda^1(\xb;u)=\{\lambda_u\}.\]
\end{proposition}
\begin{proof}
Consider an arbitrarily fixed direction $u\in{\cal C}(\xb)$ satisfying LICQ(u). We claim that
\begin{equation}\label{EqAuxProb1}\min\{\nabla f(\xb)v\mv \nabla F(\xb)v\in T(\nabla F(\xb)u;T(F(\xb);\Omega))\}=0.\end{equation}
Indeed, if there were $\bar v$ with $\nabla F(\xb)\bar v\in T(\nabla F(\xb)u;T(F(\xb);\Omega))$ and $\nabla f(\xb)\bar v<0$, then for every $\alpha>0$ sufficiently small we have $\nabla F(\xb)(u+\alpha\bar v)\in T(F(\xb);\Omega)$ and consequently $F(\xb)+t\nabla F(\xb)(u+\alpha \bar v)\in \Omega$ for all $t>0$ sufficiently small, since both $\Omega$ and $T(F(\xb);\Omega)$ are the union of finitely many polyhedra.
By Lemma \ref{LemBasLICQ} and Lemma \ref{LemSuffDirSubreg} we have that $M(x)=F(x)-\Omega$ is metrically subregular in direction $u$ at $(\xb,0)$ and therefore there are $\rho>0$, $\delta>0$, $\kappa>0$ such that for all $x\in\xb+V_{\rho,\delta}(u)$ the inequality \eqref{EqDirSubReg} holds. We can choose $\alpha>0$ small enough such that for all $t>0$ sufficiently small we have $\xb+t(u+\alpha\bar v)\in \xb+V_{\rho,\delta}(u)$ and $F(\xb)+t\nabla F(\xb)(u+\alpha \bar v)\in \Omega$ implying the existence of $w(t)$ with $F(\xb+t(u+\alpha\bar v+w(t)))\in\Omega$ and
\[t\norm{w(t)}\leq\kappa \dist{F(\xb+t(u+\alpha\bar v)),\Omega}\leq \kappa\norm{F(\xb+t(u+\alpha\bar v))-F(\xb)-t\nabla F(\xb)(u+\alpha\bar v)}=:\chi(t).\]
Since $F$ is Fr\'echet differentiable,  $\lim_{t\downarrow 0}\chi(t)/t=0$ and thus $\lim_{t\downarrow 0}w(t)=0$ and $u+\alpha\bar v\in T(\xb;{\cal F})$ follow. But $\nabla f(\xb)(u+\alpha\bar v)\leq \alpha\nabla f(\xb)\bar v<0$ contradicting B-stationarity of $\xb$ and hence our claim is proved. The constraint mapping of \eqref{EqAuxProb1} is a polyhedral multifunction and hence metrically subregular. Applying the M-stationarity condition at $v=0$ yields the existence of some multiplier $\tilde\lambda\in N(0;T(\nabla F(\xb)u; T(F(\xb);\Omega)))$ with
$\nabla f(\xb)^T+\nabla F(\xb)^T\tilde\lambda=0$. By \eqref{EqInclNormalCone3} we conclude $N(0;T(\nabla F(\xb)u; T(F(\xb);\Omega)))\subset N(F(\xb);\Omega;\nabla F(\xb)u)$ and $\tilde\lambda\in\Lambda^1(\xb;u)\not=\emptyset$ follows. Next we show that $\Lambda^1(\xb;u)$ is a singleton. Assume on the contrary that there are two different elements $\lambda^i\in\Lambda^1(\xb;u)$, $i=1,2$. By Lemma \ref{LemBasicPropCone} we have $\lambda^i\in\hat N(z^i;T(\nabla F(\xb)u; T(F(\xb);\Omega)))$ with $z^i\in T(\nabla F(\xb)u; T(F(\xb);\Omega))$ and because of $z^i+L\subset T(\nabla F(\xb)u;T(F(\xb);\Omega))$ we conclude that $\lambda^i$ belongs to $L^\perp$, $i=1,2$. Further $\nabla F(\xb)^T\lambda^i=-\nabla f(\xb)^T$,  $i=1,2$ and we obtain the contradiction
$0\not=\lambda^1-\lambda^2\in \ker \nabla F(\xb)^T\cap L^\perp=\Range \nabla F(\xb)^\perp\cap L^\perp=(\Range \nabla F(\xb)+L)^\perp={\R^m}^\perp=\{0\}.$ Since $\hat\Lambda^1(\xb;u)\subset \Lambda^1(\xb;u)$, it suffices now to show $\hat\Lambda^1(\xb;u)\not=\emptyset$ in order to prove $\Lambda^1(\xb;u)=\hat\Lambda^1(\xb;u)=\{\tilde \lambda\}$. We claim that
\begin{equation}\label{EqAuxProb2}\min\{\nabla f(\xb)v\mv \nabla F(\xb)v\in \co T(\nabla F(\xb)u;T(F(\xb);\Omega))\}=0.\end{equation}
Assume on the contrary that there is some $\bar v$ with $\nabla F(\xb)\bar v\in \co T(\nabla F(\xb)u;T(F(\xb);\Omega))$ and  $\nabla f(\xb)\bar v<0$. Then $\nabla F(\xb)\bar v$ can be represented as a convex combination $\sum_{i=1}^k\mu_iz_i$ of elements $z_1,\ldots, z_k\in T(\nabla F(\xb)u;T(F(\xb);\Omega))$. Each element $z_i$ can be written in the form $\nabla F(\xb)v_i+w_i$ with $w_i\in L$ and we obtain $\nabla F(\xb)v_i=z_i-w_i\in T(\nabla F(\xb)u;T(F(\xb);\Omega))$ and consequently $\nabla f(\xb)v_i\geq 0$ because of \eqref{EqAuxProb1}. Then, using $\tilde\lambda\in L^\perp$ we obtain the contradiction
\[0>\nabla f(\xb)\bar v=-\tilde\lambda^T\nabla F(\xb)\bar v=-\sum_{i=1}^k\mu_i\tilde\lambda^T (\nabla F(\xb) v_i+w_i)=\sum_{i=1}^k\mu_i\nabla f(\xb)v_i\geq 0.\]
Therefore \eqref{EqAuxProb2} holds true and since $\co T(\nabla F(\xb)u;T(F(\xb);\Omega))$ is a polyhedral cone as the convex hull of the union of finitely many polyhedral cones, we obtain that the constraint mapping $v\rightrightarrows\nabla F(\xb)v-\co T(\nabla F(\xb)u;T(F(\xb);\Omega))$ is metrically subregular at $(0,0)$ Applying now the M-stationarity condition at $v=0$ yields the existence of some multiplier
\begin{eqnarray*}\hat\lambda&\in& N(0;\co T(\nabla F(\xb)u; T(F(\xb);\Omega)))=\hat N(0;\co T(\nabla F(\xb)u; T(F(\xb);\Omega)))\\
&=&\hat N(0;T(\nabla F(\xb)u; T(F(\xb);\Omega)))=\hat N(\nabla F(\xb)u; T(F(\xb);\Omega))
\end{eqnarray*} with
$\nabla f(\xb)^T+\nabla F(\xb)^T\hat\lambda=0$ and therefore $\hat\lambda\in\hat\Lambda^1(\xb;u)\not=\emptyset$ and this completes the proof.
\end{proof}
\begin{corollary}
Assume that $\xb$ is B-stationary for the problem \eqref{EqMPDC}.
Then for every critical direction $u\in{\cal C}(\xb)$ fulfilling
LICQ(u) we have
\[\Lambda^{\lambda_0}(\xb;u)=\hat \Lambda^{\lambda_0}(\xb;u)\ \forall \lambda_0\geq 0.\]
\end{corollary}
\begin{proof}In case $\lambda_0=0$ we have $\Lambda^{0}(\xb;u)=\hat
\Lambda^{0}(\xb;u)=\emptyset$ because of Lemmas
\ref{LemRelationsMult}, \ref{LemBasLICQ}. If $\lambda_0>0$, the
assertion follows from the relations
$\hat\Lambda^{\lambda_0}(\xb;u)=\lambda_0\hat\Lambda^1(\xb;u)$,
$\Lambda^{\lambda_0}(\xb;u)=\lambda_0\Lambda^1(\xb;u)$ and
Proposition \ref{PropLICQ}.
\end{proof}

We now state a second-order sufficient condition in terms of multipliers belonging to $\Lambda^1(\xb;u)$.
\begin{theorem}\label{ThMSOSC}
Assume that $\xb$ is an extended M-stationary solution for \eqref{EqMPDC}, $f$ and $F$ are twice Fr\'echet differentiable at $\xb$ and that for every nonzero critical direction $0\not=u\in{\cal C}(\xb)$ one has
\begin{equation}\label{EqMSOSC}
u^T\nabla_x^2{\cal L}(\xb,1,\lambda)u>0\ \forall \lambda\in\Lambda^1(\xb;u).
\end{equation}
Then $\xb$ is an essential local minimizer of second order.
\end{theorem}
\begin{proof}
By contraposition. Assuming on the contrary that $\xb$ is not an essential local minimizer, by Lemma \ref{LemCharEssLocMin} we can find $0\not=u\in{\cal C}(\xb)$ and $v\in\R^n$ fulfilling \eqref{EqEssMinCond1}, \eqref{EqEssMinCond2}. We now claim that the problem
\begin{equation}\label{EqAuxProb3}\min_v\nabla f(\xb)v\ \mbox{subject to}\ \nabla F(\xb)v+\frac 12 u^T\nabla^2 F(\xb)u\in T(\nabla F(\xb)u; T(F(\xb);\Omega))
\end{equation}
has an optimal solution. If there would not exist an optimal solution, because the feasible region is not empty because of \eqref{EqEssMinCond2}, we could find a sequence $(v^k)$ feasible for \eqref{EqAuxProb3} such that $\nabla f(\xb)v^k\to -\infty$. Consider the sequence $\tilde v^k:=v^k/\vert\nabla f(\xb)v^k\vert$. Then
$\dist{\nabla F(\xb)\tilde v^k,T(\nabla F(\xb)u; T(F(\xb);\Omega))}\to 0$ and since $v\rightrightarrows\nabla F(\xb)v-T(\nabla F(\xb)u; T(F(\xb);\Omega))$ is a polyhedral multifunction and therefore metrically subregular at $(0,0)$, there is a sequence $(\hat v^k)$ with $\nabla F(\xb)\hat v^k\in T(\nabla F(\xb)u;T(F(\xb);\Omega))$ and $\lim_{k\to\infty} (\tilde v^k-\hat v^k)=0$. Fixing $\bar v:=\hat v^k$ for $k$ sufficiently large we have $\nabla f(\xb)\bar v<-\frac 12$. Since $\nabla F(\xb)(u+\alpha\bar v)\in T(F(\xb);\Omega)$ for $\alpha>0$ sufficiently small we have $u+\alpha\bar v\in T_{\rm lin}(\xb)$. Together with $\nabla f(\xb)(u+\alpha\bar v)<-\frac{\alpha}2<0$ we have $u+\alpha\bar v\in{\cal C}(\xb)$ and thus $\Lambda^1(\xb;u+\alpha\bar v)\not=\emptyset$ by extended M-stationarity of $\xb$. But from \eqref{EqLambda0Nabla_f} we obtain the contradiction $\nabla f(\xb)(u+\alpha\bar v)=0$. Hence the problem \eqref{EqAuxProb3} has an optimal solution $\tilde v$. Since the constraint mapping is a polyhedral multifunction and therefore metrically subregular at $(\tilde v,0)$, we can apply the M-stationarity conditions at $\tilde v$ to find a multiplier
\[\lambda\in N(\nabla F(\xb)\tilde v+\frac 12 u^T\nabla^2 F(\xb)u; T(\nabla F(\xb)u;T(F(\xb);\Omega)))\]
with $\nabla f(\xb)^T+\nabla F(\xb)^T\lambda=\nabla_x{\cal L}(\xb;1,\lambda)^T=0$, showing, together with $\lambda\in N(F(\xb);\Omega;\nabla F(\xb)u)$ because of \eqref{EqInclNormalCone3},  $\lambda\in\Lambda^1(\xb;u)$. Using extended M-stationarity of $\xb$ and \eqref{EqLambda0Nabla_f} we obtain $\nabla f(\xb)u=0$ and therefore $\nabla f(\xb)\tilde v+\frac 12 u^T\nabla^2 f(\xb)u\leq 0$ because of \eqref{EqEssMinCond1}. Because $T(\nabla F(\xb)u;T(F(\xb);\Omega))$ is a cone we  have $\lambda^T(\nabla F(\xb)\tilde v+\frac 12 u^T\nabla^2 F(\xb)u)=0$ and thus
\begin{eqnarray*}
0&\geq&\nabla f(\xb)\tilde v+\frac 12 u^T\nabla f(\xb)u +\lambda^T(\nabla F(\xb)\tilde v+\frac 12 u^T\nabla^2 F(\xb)u)\\
&=&\nabla_x{\cal L}(\xb,1,\lambda)\tilde v+\frac 12u^T\nabla_x^2{\cal L}(\xb,1,\lambda)u=\frac 12u^T\nabla_x^2{\cal L}(\xb,1,\lambda)u
\end{eqnarray*}
contradicting \eqref{EqMSOSC}.
\end{proof}
\begin{remark}
Following \cite[Definition 3.2]{GuoLinYe12a} the point $\xb$ is said to fulfill the {\em strong second-order sufficient condition} (SSOSC) for \eqref{EqMPDC} if $\Lambda^1(\xb)\not=\emptyset$ and  for every nonzero critical direction $0\not=u\in{\cal C}(\xb)$ one has
\[u^T\nabla_x^2{\cal L}(\xb,1,\lambda)u>0\ \forall \lambda\in\Lambda^1(\xb).\]
However note that this condition is not sufficient for $\xb$ to be a local minimizer as can be easily seen from the example
\[\min-x_1+x_1^2+x_2^2\ \mbox{subject to}\ (-x_1,-x_2)\in Q_{\rm EC}.\]
In order to make (SSOSC) sufficient for $\xb$ being a local minimizer, in view of Theorem \ref{ThMSOSC} we have to replace the M-stationarity condition $\Lambda^1(\xb)\not=\emptyset$ by the extended M-stationarity condition $\Lambda^1(\xb;u)\not=\emptyset$ $\forall 0\not=u\in{\cal C}(\xb)$.
\end{remark}

\section{Applications to MPECs\label{SecMPEC}}
We now want to apply the results of the preceding section to the
MPEC \eqref{EqMPEC}, or more exactly, to the problem
\eqref{EqMPDC} with $F$ and $\Omega$ given by \eqref{EQMPECData}.
By straightforward calculation we can obtain the formulas for the
Fr\'echet normal cone, the Mordukhovich normal cone and the
contingent cone of the set $Q_{\rm EC}$ defined in \eqref{EqQEC}
as follows:
\begin{lemma} For all $a=(a_1,a_2)\in Q_{\rm EC}$ we have
\[\hat N(a;\Omega_{\rm EC})=\left\{(\xi_1,\xi_2)\mv\begin{array}{ll} \xi_2=0&\mbox{if $0=a_1>a_2$}\\
\xi_1\geq0,\xi_2\geq 0&\mbox{if $a_1=a_2=0$}\\
\xi_1=0&\mbox{if $a_1<a_2=0$}
\end{array}\right\},\]
\[N(a;\Omega_{\rm EC})=\begin{cases}\hat N(a;\Omega_{\rm EC})&\mbox{if $a\not=(0,0)$}\\
\{(\xi_1,\xi_2)\mv\mbox{either $\xi_1>0$, $\xi_2>0$ or $\xi_1\xi_2=0$}\}&\mbox{if $a=(0,0)$},
\end{cases}\]
\[T(a;\Omega_{\rm EC})=\left\{(u_1,u_2)\mv\begin{array}{ll}u_1=0&\mbox{if $0=a_1>a_2$}\\
u_1\leq 0,u_2\leq0, u_1u_2=0&\mbox{if $a_1=a_2=0$}\\
u_2=0&\mbox{if $a_1<a_2=0$}
\end{array}\right\}\]
and for all $u=(u_1,u_2)\in T(a;\Omega_{\rm EC})$ we have
\[T(u;T(a;\Omega_{\rm EC}))=\begin{cases}T(a;\Omega_{\rm EC})&\mbox{if $a\not=(0,0)$}\\
T(u;\Omega_{\rm EC}))&\mbox{if $a=(0,0)$},
\end{cases}\]
\[\hat N(u;T(a;\Omega_{\rm EC}))=\begin{cases}\hat N(a;\Omega_{\rm EC})&\mbox{if $a\not=(0,0)$}\\
\hat N(u;\Omega_{\rm EC})&\mbox{if $a=(0,0)$},
\end{cases}\]
\[N(a;\Omega_{\rm EC};u)=\begin{cases}N(a;\Omega_{\rm EC})&\mbox{if $a\not=(0,0)$}\\
 N(u;\Omega_{\rm EC})&\mbox{if $a=(0,0)$}.
\end{cases}\]
\end{lemma}

In what follows, we denote by $\xb$ a point feasible for the MPEC \eqref{EqMPEC}. Further we assume throughout this section that the mappings $f$, $g$, $h$, $G$, $H$ are continuously Fr\'echet differentiable, twice Fr\'echet differentiable at $\xb$ and that there are numbers $1\leq l_1\leq l$, $1\leq p_1\leq p$, $1\leq q_1\leq q$ such that the components
\[ g_i(x), i=l_1+1,\ldots,l,\quad h_i(x), i=p_1+1,\ldots,p,\quad G_i(x),H_i(x), i=q_1+1\ldots,q\]
are affine or linear. In what follows, for every direction $u\in T_{\rm lin}(\xb)$ the multifunction $M_2$, which is assumed to be metrically subregular in direction $u$, is build by the linear parts of the constraints.

Denoting
\begin{eqnarray*}
&&\bar I_g:=\{i\in\{1,\ldots,l\}\mv g_i(\xb)=0\},\\
&&\bar I^{+0}:=\{i\in\{1,\ldots,q\}\mv G_i(\xb)>0=H_i(\xb)\},\\
&&\bar I^{0+}:=\{i\in\{1,\ldots,q\}\mv G_i(\xb)=0<H_i(\xb)\},\\
&&\bar I^{00}:=\{i\in\{1,\ldots,q\}\mv G_i(\xb)=0=H_i(\xb)\},
\end{eqnarray*}
the cone $T_{\rm lin}(\xb)$ is given by
\[T_{\rm lin}(\xb)=\left\{u\in\R^n\mv \begin{array}{l}\nabla g_i(\xb)u\leq 0,\ i\in\bar I_g,\\
\nabla h_i(\xb)u=0,\ i=1,\ldots,p,\\
\nabla G_i(\xb)u=0,\ i\in \bar I^{0+},\\
\nabla H_i(\xb)u=0,\ i\in \bar I^{+0},\\
-(\nabla G_i(\xb)u,\nabla H_i(\xb)u)\in Q_{\rm EC},\ i\in \bar I^{00}
\end{array}\right\}.\]

The generalized Lagrangian reads as
\[{\cal L}(x,\lambda_0,\lambda)=\lambda_0f(x)+{\lambda^g}^Tg(x)+{\lambda^h}^Th(x)-{\lambda^G}^TG(x)-{\lambda^H}^TH(x),\]
where $\lambda_0\in\R$, $\lambda:=(\lambda^g,\lambda^h,\lambda^G,\lambda^H)\in\R^l\times\R^p\times\R^q\times\R^q$.

Given $u\in T_{\rm lin}(\xb)$ we define
\begin{eqnarray*}&&I_g(u):=\{i\in\bar I_g\mv\nabla g_i(\xb)u=0\}\\
&&I^{+0}(u):=\{i\in \bar I^{00}\mv \nabla G_i(\xb)u>0=\nabla H_i(\xb)u\},\\
&&I^{0+}(u):=\{i\in \bar I^{00}\mv \nabla G_i(\xb)u=0<\nabla H_i(\xb)u\},\\
&&I^{00}(u):=\{i\in \bar I^{00}\mv \nabla G_i(\xb)u=0=\nabla H_i(\xb)u\}.
\end{eqnarray*}
Then for $\lambda_0\geq 0$ we have
\[\Lambda^{\lambda_0}(\xb;u)=\left\{\lambda=(\lambda^g,\lambda^h,\lambda^G,\lambda^H)\mv\begin{array}{l}\nabla_x{\cal L}(\xb,\lambda_0,\lambda)=0\\
\lambda^g_i\geq 0,\lambda^g_i g_i(\xb)=0,\ i\in\{1,\ldots, l\}\\
\lambda^g_i=0,\ i\in\bar I_g\setminus I_g(u)\\
\lambda^H_i=0,\ i\in \bar I^{0+}\cup I^{0+}(u)\\
\lambda^G_i=0,\ i\in \bar I^{+0}\cup I^{+0}(u)\\
\mbox{either $\lambda^G_i>0,\lambda^H_i>0$ or $\lambda^G_i\lambda^H_i=0$},\ i\in I^{00}(u)\\
\lambda_0+\sum_{i=l}^{l_1}\lambda^g_i+\sum_{i=1}^{p_1}\vert \lambda^h_i\vert+\sum_{i=1}^{q_1}(\vert \lambda^G_i\vert+\vert \lambda^H_i\vert)>0
\end{array}\right\},\]
\[\hat\Lambda^{\lambda_0}(\xb;u)=\left\{\lambda=(\lambda^g,\lambda^h,\lambda^G,\lambda^H)\mv\begin{array}{l}\nabla_x{\cal L}(\xb,\lambda_0,\lambda)=0\\
\lambda^g_i\geq 0,\lambda^g_i g_i(\xb)=0,\ i\in\{1,\ldots, l\}\\
\lambda^g_i=0,\ i\in\bar I_g\setminus I_g(u)\\
\lambda^H_i=0,\ i\in \bar I^{0+}\cup I^{0+}(u)\\
\lambda^G_i=0,\ i\in \bar I^{+0}\cup I^{+0}(u)\\
\lambda^G_i\geq0,\lambda^H_i\geq0,\ i\in I^{00}(u)\\
\lambda_0+\sum_{i=l}^{l_1}\lambda^g_i+\sum_{i=1}^{p_1}\vert \lambda^h_i\vert+\sum_{i=1}^{q_1}(\vert \lambda^G_i\vert+\vert \lambda^H_i\vert)>0
\end{array}\right\}.\]
Since
\[T(\nabla F(\xb)u;T(F(\xb);\Omega))=\left\{
(z^g,z^h,z_1^G,z_1^H,\ldots,z_q^G,z_q^H)^T\in\R^l\times \R^p\times
\R^{2q}\mv \begin{array}{l}z_i^g\leq 0,\ i\in
I^g(u),\\
z^h=0,\\
z_i^G=0,\ i\in \bar I^{0+}\cup I^{0+}(u),\\
z_i^H=0,\ i\in \bar I^{+0}\cup I^{+0}(u),\\
(z_i^G,z_i^H)\in Q_{EC},\ i\in I^{00}(u)
\end{array}\right\},\]
the largest possible subspace $L$ such that $T(\nabla
F(\xb)u;T(F(\xb);\Omega))+L\subset T(\nabla
F(\xb)u;T(F(\xb);\Omega))$ is given by
\[L=\left\{
(z^g,z^h,z_1^G,z_1^H,\ldots,z_q^G,z_q^H)^T\in\R^l\times \R^p\times
\R^{2q}\mv \begin{array}{l}z_i^g= 0,\ i\in
I^g(u),\\
z^h=0,\\
z_i^G=0,\ i\in \bar I^{0+}\cup I^{0+}(u),\\
z_i^H=0,\ i\in \bar I^{+0}\cup I^{+0}(u),\\
z_i^G=z_i^H=0,\ i\in I^{00}(u)
\end{array}\right\}.\]
Hence LICQ(u) is fulfilled if and only if the family of gradients
\begin{eqnarray*}&&\{\nabla g_i(\xb)\mv i\in I_g(u)\}\cup\{\nabla h_i(\xb)\mv i\in\{1,\ldots,p\}\}
\cup\{\nabla G_i(\xb)\mv i\in \bar I^{0+}\cup I^{0+}(u)\cup I^{00}(u)\}\\
&&\cup
\{\nabla H_i(\xb)\mv i\in \bar I^{+0}\cup I^{+0}(u)\cup I^{00}(u)\}
\end{eqnarray*}
is linearly independent. It is easy to see that LICQ(0) is exactly the well-known MPEC LICQ condition.

\begin{example}\label{ExMStatNotLocMin}
Consider the problem
\begin{eqnarray*}\min_{x=(x_1,x_2,x_3)} f(x)&:=&x_1+x_2-2x_3\\
g_1(x)&:=& -x_1-x_3\leq 0,\\
g_2(x)&:=& -x_2+x_3\leq 0,\\
-(G_1(x),H_1(x))&:=&-(x_1,x_2)\in Q_{\rm EC}.
\end{eqnarray*}
Then $\xb=(0,0,0)$ is not a local minimizer, because for every $\alpha>0$ the point $x^\alpha:=(0,\alpha,\alpha)$ is feasible and $f(x^\alpha)=-\alpha<0$. Indeed, for the critical direction $u=(0,1,1)$ we have $\Lambda^1(\xb;u)=\emptyset$ and therefore $\xb$ is not an extended M-stationary solution and consequently not a local minimizer, since all problem functions are linear and the constraint mapping is thus metrically subregular. However, $\xb$ is M-stationary since $\Lambda^1(\xb)=\{(1,3,0,-2)\}$.
\end{example}

To demonstrate the results on second-order optimality conditions of the preceding section we consider the following example:
\begin{example}
Consider the parameter dependent problem
\begin{eqnarray*}P(a)\qquad \min_{x_1,x_2} f(x_1,x_2)&:=&-x_1+\frac 12 x_2^2\\
 g_1(x_1,x_2)&:=& ax_1^2-x_2\leq 0,\\
-(G_1(x_1,x_2),H_1(x_1,x_2))&:=&-(x_1,x_2)\in Q_{\rm EC}
\end{eqnarray*}
where $a\in\R$. Then it is easy to see that $\xb=(0,0)$  is a local minimizer, if and only if
$a>0$. Let us verify this by using our theory.

We have
\[{\cal L}(x,\lambda_0,(\lambda^g,\lambda^G,\lambda^H))=\lambda_0(-x_1+\frac 12 x_2^2)+\lambda^g(ax_1^2-x_2)-\lambda^Gx_1-\lambda^Hx_2\]
and for every $a$ it follows that $\Lambda^0(\xb)=\{(\alpha,0,-\alpha\mv\alpha>0\}$, implying that metric regularity of the constraint mapping and therefore also LICQ(0) are violated.
Further we have
\[T_{\rm lin}(\xb)=\left\{(u_1,u_2)\mv -u_2\leq0,\ (-u_1,-u_2)\in Q_{\rm EC}\right\}=-Q_{\rm EC}\]
and
${\cal C}(\xb)=T_{\rm lin}(\xb)$, i.e. we have to analyze the problem with respect to the two critical directions $(1,0)$ and $(0,1)$.
\begin{enumerate}
\item $u=(1,0)$: Then $\Lambda^1(\xb;u)=\emptyset$ and $\Lambda^0(\xb;u)=\hat\Lambda^0(\xb;u)=\Lambda^0(\xb)$ and taking $\lambda=(\alpha,0,-\alpha)$ with $\alpha>0$ we have
\begin{equation}\label{EqFirstCritDir}u^T\nabla_x^2{\cal L}(\xb,0,\lambda)u=2\alpha a\begin{cases}
<0&\mbox{if $a<0$,}\\>0&\mbox{if $a>0$}
\end{cases}.\end{equation}
By the second-order conditions \eqref{EqSecOrdNec} we conclude that $\xb$ is not a local minimizer for $a<0$. In case $a=0$ the constraint mapping is polyhedral and hence metrically subregular. Since $\Lambda^1(\xb,u)=\emptyset$ we can also conclude from Theorem \ref{ThBaseOptCond} that $\xb$ is not a local minimizer in case $a=0$.
\item $u=(0,1)$: In this case LICQ(u) is fulfilled and we have $\hat\Lambda^1(\xb;u)=\Lambda^1(\xb;u)=\{(0,-1,0)\}$. Since
$u^T\nabla_x^2{\cal L}(\xb,1,(0,-1,0))u=1>0$, together with \eqref{EqFirstCritDir}, we conclude from Theorem \ref{ThEssLocMin} that $\xb$ is a essential local minimizer of second order in case $a>0$.

\end{enumerate}
\end{example}

Since extended M-stationarity is usually difficult to verify in practice, we now introduce the following concept of {\em strong M-stationarity}, which builds a bridge between M-stationarity and S-stationarity. In what follows we note by $r(\xb)$ the rank of the family of gradients
\begin{equation}\label{EqActGrad}\{\nabla g_i(\xb)\mv i\in \bar I_g\}\cup\{\nabla h_i(\xb)\mv i\in\{1,\ldots,p\}\}\cup\{\nabla G_i(\xb)\mv i\in \bar I^{0+}\cup \bar I^{00}\}\cup
\{\nabla H_i(\xb)\mv i\in \bar I^{+0}\cup \bar I^{00}\}.
\end{equation}
\begin{definition}
\begin{enumerate}
\item A triple of index sets $(J_g, J_G, J_H)$, $J_g\subset \bar I_g$, $J_G\subset \bar I^{0+}\cup\bar I^{00}$, $J_H\subset \bar I^{+0}\cup\bar I^{00}$
 is called a {\em MPEC working set} for the MPEC \eqref{EqMPEC}, if $J_G\cup J_H=\{1,\ldots,q\}$,
\[\vert J_g\vert +p+\vert J_G\vert+\vert J_H\vert=r(\xb)\]
and the family of gradients
\[\{\nabla g_i(\xb)\mv i\in J_g\}\cup\{\nabla h_i(\xb)\mv i\in\{1,\ldots,p\}\}\cup\{\nabla G_i(\xb)\mv i\in J_G\}\cup
\{\nabla H_i(\xb)\mv i\in J_H\}\]
is linearly independent.
\item The point $\xb$ is called {\em strongly M-stationary} for the MPEC \eqref{EqMPEC}, if there is a MPEC working set $(J_g,J_G,J_H)$ together with a multplier $\lambda=(\lambda^g,\lambda^h,\lambda^G,\lambda^H)\in \Lambda^1(\xb)$ satisfying
\begin{eqnarray}
\label{EqLambdaZero_g}&&\lambda_i^g=0,\ i\in\{1,\ldots,l\}\setminus J_g,\\
\label{EqLambdaZero_G}&&\lambda_i^G=0,\ i\in\{1,\ldots,q\}\setminus J_G,\\
\label{EqLambdaZero_H}&&\lambda_i^H=0,\ i\in\{1,\ldots,q\}\setminus J_H,\\
\label{EqComplCondActSet}&&\lambda_i^G\geq0, \lambda_i^H\geq0,\ i\in J_G\cap J_H.
\end{eqnarray}
\end{enumerate}
\end{definition}
Note that the condition $J_G\cup J_H=\{1,\ldots,q\}$ implies $\bar I^{0+}\subset J_G$ and $\bar I^{+0}\subset J_H$. By the definition, every strongly M-stationary point is M-stationary. However, the converse is not true as can be seen from the example \[\min -x_1\mbox{ subject to } -(x_1,x_2)\in Q_{EC},\] where $\xb=(0,0)$ is M-stationary but not strongly M-stationary.
\begin{theorem}\label{ThStrongMStat}
Assume that $\xb$ is extended M-stationary for the problem
\eqref{EqMPDC} with $F$ and $\Omega$ given by \eqref{EQMPECData}
and assume that there exists some MPEC working set. Then $\xb$ is
strongly M-stationary.
\end{theorem}
\begin{proof}
Since $\xb$ is extended M-stationary, we have $\Lambda^1(\xb)\not=\emptyset$ and therefore $\nabla f(\xb)$ can be represented as a linear combination of the gradients \eqref{EqActGrad}. It follows that for every MPEC working set $J=(J_g,J_G,J_H)$ there is a unique multiplier $\lambda(J)=(\lambda^g,\lambda^h,\lambda^G,\lambda^H)$ satisfying \eqref{EqLambdaZero_g}-\eqref{EqLambdaZero_H} and $\nabla_x{\cal L}(\xb,1,\lambda(J))=0$. Now let $J^0=(J_g^0,J_G^0,J_H^0)$ be an arbitrarily fixed working set and choose $b=(b^g,b^G,b^H)\in\R^l_+\times\R^q_-\times\R^q_-$ with
$b^g_i=0$, $i\in J_g^0$, $b^G_i=0$, $i\in J_G^0$ and $b^H_i=0$, $i\in J_H^0$ such that for all $u \in \R^n$ the family of gradients
\begin{eqnarray}\nonumber&\{\nabla g_i(\xb)\mv i\in \bar I_g, \nabla g_i(\xb)u=b^g_i\}\cup\{\nabla h_i(\xb)\mv i\in\{1,\ldots,p\}, \nabla h_i(\xb)u=0\}&\\
\label{EqLinDep1}&\cup\{\nabla G_i(\xb)\mv i\in \bar I^{0+}\cup \bar I^{00},\nabla G_i(\xb)u=b^G_i\}\cup
\{\nabla H_i(\xb)\mv i\in \bar I^{+0}\cup \bar I^{00},\nabla H_i(\xb)u=b^H_i\}&
\end{eqnarray}
is linearly independent. Such a vector $b$ exists by the following arguments. For every triple of index sets $K=(K_g,K_G,K_H)$, $K_g\subset \bar I^g$, $K_G\subset \bar I^{0+}\cup\bar I^{00}$, $K_H\subset \bar I^{+0}\cup\bar I^{00}$ let ${\cal B}(K)$ denote a basis for the subspace
\[\{\mu=\{(\mu^g,\mu^h,\mu^G,\mu^H)\mv \begin{array}{l}\nabla_x{\cal L}(\xb,0,\mu)=0,\\
\mu_i^g=0,\ i\in\{1,\ldots,l\}\setminus K_g,\\
\mu_i^G=0,\ i\in\{1,\ldots,q\}\setminus K_G,\\
\mu_i^H=0,\ i\in\{1,\ldots,q\}\setminus K_H
\end{array}\},\]
where ${\cal B}(K)$ is eventually empty. By the definition of a MPEC working set, for every basis element $\mu=(\mu^g,\mu^h,\mu^G,\mu^H)$ there must be either
an index $i\in \bar I_g\setminus J_g^0$ with $\mu^g_i\not=0$ or an index $i\in (\bar I^{0+}\cup\bar I^{00})\setminus J_G^0$ with $\mu^G_i\not=0$
or an index $i\in (\bar I^{+0}\cup\bar I^{00})\setminus J_H^0$ with $\mu^H_i\not=0$.
The union of the bases $\bigcup_K{\cal B}(K)$ consists of finitely many elements and therefore we can find $b=(b^g,b^G,b^H)\in\R^l_+\times\R^q_-\times\R^q_-$ with $b^g_i=0$, $i\in J_g^0$, $b^G_i=0$, $i\in J_G^0$ and $b^H_i=0$, $i\in J_H^0$ with
\[{b^g}^T\mu^g+{b^G}^T\mu^G+{b^H}^T\mu^H\not=0\ \forall (\mu^g,\mu^h,\mu^G,\mu^H)\in\bigcup_K{\cal B}(K)\]
We claim that this vector $b$ has the required property. If there would exist $u\in\R^n$ such that \eqref{EqLinDep1} does not hold, by taking $K_g:=\{i\in \bar I_g\mv \nabla g_i(\xb)u=b^g_i\}$, $K_G:=\{i\in \bar I^{0+}\cup \bar I^{00}\mv \nabla G_i(\xb)u=b^G_i\}$, $K_H:=\{i\in \bar I^{+0}\cup \bar I^{00}\mv\nabla H_i(\xb)u=b^H_i\}$, there is some element $\mu=(\mu^g,\mu^h,\mu^G,\mu^H)\in{\cal B}(K_g,K_G,K_H)$ with
\[0=\nabla_x{\cal L}(\xb,0,\mu)u={b^g}^T\mu^g+{b^G}^T\mu^G+{b^H}^T\mu^H\not=0,\]
a contradiction, and therefore our claim is proved. Now consider the following algorithm:

\noindent\fbox{\parbox{\hsize}{\parindent2em\noindent
\Itl1{1:} $u:=0$, $J:=J^0$, $(\lambda^g,\lambda^h,\lambda^G,\lambda^H):=\lambda(J)$;
\Itl1{2:} {\tt while} $((\exists i\in J_g:\lambda^g_i<0)\ \vee\ (\exists i\in J_G\cap J_H: \lambda^G_i<0 \vee \lambda^H_i<0))$
\Itlb2{3:}{\tt if} $(\exists i_0\in J_g:\lambda^g_{i_0}<0)$
\Itl3{4:} $J_g:=J_g\setminus\{i_0\}$;
\Itl2{5:} {\tt else}
\Itlb3{6:} select $i_0\in J_G\cap J_H$ with $\lambda^G_{i_0}<0$ or $\lambda^H_{i_0}<0$;
\Itl3{7:}{\tt if }$(\lambda^G_{i_0}<0)$
\Itl4{8:} $J_G:=J_G\setminus\{i_0\}$;
\Itl3{9:} {\tt else}
\Itl4{10:} $J_H:=J_H\setminus\{i_0\}$;
\Itl2{11:}$\}$
\Itl2{12:} Compute search direction $d$ with $\nabla f(\xb)d=-1$, $\nabla g_i(\xb)d=0,i\in J_g$,
\It3 $\nabla h_i(\xb)d=0,i=1,\ldots,p$, $\nabla G_i(\xb)d=0,i\in J_G$, $\nabla H_i(\xb)d=0, i\in J_H$;
\Itl2{13:} Compute step length
\[\hat\alpha_j=\min\{\min_{\AT{i\in\bar I_g\setminus J_g}{\nabla g_i(\xb) d>0}}\{\frac{b_i^g-\nabla g_i(\xb)u}{\nabla g_i(\xb) d}\}, \min_{\AT{i\in \bar I^{00}\setminus J_G}{\nabla G_i(\xb)d<0}}\{\frac{b^G_i-\nabla G_i(\xb)u}{\nabla G_i(\xb)d}\}, \min_{\AT{i\in \bar I^{00}\setminus J_H}{\nabla H_i(\xb)d<0}}\{\frac{b^H_i-\nabla H_i(\xb)u}{\nabla H_i(\xb)d}\}\};\]
\Itl2{14:}//The index $j$ indicates the constraint to enter the MPEC working set
\Itl2{15:} Either set $J_g:=J_g\cup\{j\}$ or $J_G:=J_G\cup\{j\}$ or $J_H:=J_H\cup\{j\}$, depending in which part
\Itl3{} the minimum is attained when computing $\hat \alpha_j$;
\Itl2{16:} $u:=u+\hat\alpha_j d$, compute $(\lambda^g,\lambda^h,\lambda^G,\lambda^H):=\lambda(J)$;
\Itl1{17:}\}
}}\\[1ex]

This algorithm is very close to the well-known pivoting algorithms from linear programming. It can be also considered as a so-called {\em  active set method}, we refer the reader to \cite{Flet81} for an introduction to this method.

At the beginning of each cycle $(J_g,J_G,J_H)$  constitutes a MPEC working set and  we have
\begin{eqnarray*}
&&\nabla g_i(\xb)u\leq b_i^g,\ i\in \bar I_g,\\
&&\nabla h_i(\xb)u= 0,\ i=1,\ldots,p,\\
&&\nabla G_i(\xb)u= b_i^G=0,\ i\in \bar I^{0+},\\
&&\nabla G_i(\xb)u\geq b_i^G,\ i\in \bar I^{00}\\
&&\nabla H_i(\xb)u= b_i^H= 0,\ i\in \bar I^{+0},\\
&&\nabla H_i(\xb)u\geq b_i^H,\ i\in \bar I^{00}\\
&&(\nabla G_i(\xb)u-b_i^G)(\nabla H_i(\xb)u-b_i^H)=0,\ i\in \bar I^{00}
\end{eqnarray*}
and
\[J_g=\{i\in\bar I_g\mv \nabla g(\xb)u=b_i^g\},\ J_G=\{i\in\{1,\ldots q\}\mv \nabla G_i(\xb)u=b_i^H\},\ J_H=\{i\in\{1,\ldots q\}\mv \nabla H_i(\xb)u=b_i^H\}.\]
The computation of the search direction $d$ in line 12 is possible, because after removing index $i_0$ from the MPEC working set the family of  gradients
\[\{\nabla f(\xb)\}\cup\{\nabla g_i(\xb)\mv i\in J_g\}\cup \{\nabla h_i(\xb)\mv i=1,\ldots,p\}\cup \{\nabla G_i(\xb)\mv i\in J_G\}\cup \{\nabla H_i(\xb)\mv i\in J_H\}\]
is linearly independent. The minimum when computing $\hat \alpha_j$ must be attained, because otherwise the direction $d$ would fulfill $d\in T_{\rm lin}(\xb)$ and $\nabla f(\xb)d<0$ contradicting extended M-stationarity of $\xb$. Further, our construction of $b$ guarantees that the index $j$ is unique and $\hat \alpha_j$ is strictly positive.

Since the value $\nabla f(\xb)u$ strictly decreases in each cycle and only a finite number of MPEC working sets exist, the algorithm always terminates in a finite number of steps and the outcome $J=(J_g,J_G,J_H)$ together with $\lambda(J)$ proves strong M-stationarity of $\xb$.
\end{proof}

The algorithm used in the proof can be implemented in practice to
test whether a feasible point $\xb$ is strongly M-stationary or
not. In such an implementation the proper choice of $b$ is
crucial. A random choice of $b$ with
\[b_i^g>0, i\in \bar I\setminus J_g^0,\ b_i^G<0,\ i\in \bar I^{00}\setminus J_G^0,\ b_i^H<0,\ i\in \bar I^{00}\setminus J_H^0\]
and fixing the other components to $0$ will yield  a suitable
vector $b$ with probability 1,  as can be easily seen from the
arguments used in the proof. Moreover, the unlikely case of a
wrong choice of $b$ can be easily detected during the course of
the algorithm and then we can modify $b$ to meet the requirements.
Of course, one has to implement an exit in case that
$\hat\alpha_j=\infty$, i.e. $\{i\in\bar I_g\setminus J_g\mv\nabla
g_i(\xb) d>0\}=\{i\in \bar I^{00}\setminus J_G\mv\nabla
G_i(\xb)d<0\}=\{i\in \bar I^{00}\setminus J_H\mv\nabla
H_i(\xb)d<0\}=\emptyset$, since then the computed direction $d$ is
a descent direction.

In the following example we show the process of the algorithm given in
the proof of Theorem \ref{ThStrongMStat}.
\begin{example}
Consider the MPEC of Example \ref{ExMStatNotLocMin}. Then
$r(\xb)=3$ and we start the algorithm with the MPEC working set
$J^0_g:=\{1\}$, $J^0_G:=J^0_H:=\{1\}$ and $b_2^g=1$,
$b_1^g=b_1^G=b_1^H=0$, $J:=J^0$, $u=0$, resulting in
$\lambda(J)=(-2,0,3,1)$. Since $\lambda_1^g<0$, the first
inequality  leaves $J_g$ yielding $J_g=\emptyset$. Then the
direction $d=(0,0,\frac 12)^T$ is computed as the unique solution
of
\[\nabla f(\xb)d=d_1+d_2-2d_3=-1,\ \nabla G_1(\xb)d=d_1=0,\ \nabla
H_1(\xb)d=d_2=0.\] We have $\bar I^{00}=J_G=J_H$ and $\nabla
g_1(\xb)d=-\frac 12$, $\nabla g_2(\xb)d=\frac 12$ and therefore
the step length $\hat\alpha_j$ amounts to
\[\hat\alpha_j=\frac{b_2^g-\nabla g_2(\xb)u}{\nabla g_2(\xb)d}=2.\]
Next we set $J_g:=J_g\cup\{2\}=\{2\}$ and compute
$u:=u+\hat\alpha_j(0,0,\frac 12)^T=(0,0,1)^T$ and
$\lambda(J)=(0,2,1,-1)$.

The condition of the while loop is again not fulfilled because of
$1\in J_G\cap J_H $, $\lambda_1^H<0$ and hence we must start a new
cycle. The index 1 leaves $J_H$ and thus $J_H=\emptyset$. The
search direction $d:=(0,1,1)$ is now computed by
\[\nabla f(\xb)d=d_1+d_2-2d_3=-1,\ \nabla g_2(\xb)d:=-d_2+d_3=0,\ \nabla G_1(\xb)d=d_1=0.\]
Since $\nabla g_1(\xb)d=-1$,  $\bar I^{00}=J_G$ and $\nabla
H_1(\xb)d = 1$, we obtain $\hat\alpha_j=\infty$ and by the
comments above we  stop the algorithm because $d$ is a feasible
descent direction proving the non-optimality of $\xb$.
\end{example}

The assumption, that one MPEC working set  exists, is fulfilled,
if there are index sets $\tilde J_G,\tilde J_H\subset \bar I^{00}$
with $\tilde J_G\cup\tilde J_H= \bar I^{00}$ such that the family
of gradients
\begin{equation}\label{EqLinIndepWorkSet}\{\nabla h_i(\xb)\mv
i=1,\ldots,p\}\cup\{\nabla G_i(\xb)\mv i\in \bar I^{0+} \cup\tilde
J_G\}\cup\{\nabla H_i(\xb)\mv i\in \bar I^{+0}\cup\tilde
J_H\}\end{equation}
 is linearly independent and this seems to be a
rather weak assumption.  It is e.g. fulfilled if $q_1=q$ (i.e., we
treat all functions occurring in the complementarity conditions as
nonlinear functions) and for one direction $u\in T_{\rm lin}(\xb)$
the first-order condition for directional metric subregularity
$\Lambda^0(\xb;u)=\emptyset$ is fulfilled. To see this, choose
$\tilde J_G=I^{0+}(u)\cup I^{00}(u)$, $\tilde J_H=I^{+0}(u)$. Then
the family of gradients \eqref{EqLinIndepWorkSet} must be lineraly
independent, since otherwise there is a nontrivial linear
combination
\[\sum_{i=1}^p\lambda_i^h\nabla h_i(\xb)+\sum_{i\in\bar I^{0+}\cup I^{0+}(u)\cup
I^{00}(u)}\lambda_i^G\nabla G_i(\xb)+\sum_{i\in\bar I^{+0}\cup
I^{+0}(u)}\lambda_i^H\nabla H_i(\xb)=0\] resulting in $0$ and
hence, by setting $\lambda_i^g:=0$, $i=1,\ldots, l$,
$\lambda_i^G:=0$, $i\in \bar I^{+0}\cup I^{+0}(u)$,
$\lambda_i^H:=0$, $i\in \bar I^{0+}\cup I^{0+}(u)\cup I^{00}(u)$
and therefore $\lambda_i^G\lambda_i^H=0$, $i\in I^{00}(u)$, we
would obtain $0\not=(\lambda^g,\lambda^h,\lambda^G,\lambda^H)\in
\Lambda^0(\xb;u)$.

The following theorem justifies the definition of strongly M-stationary solutions.
\begin{theorem}
Let $\xb$ be feasible for \eqref{EqMPEC} and assume  that LICQ(0)
is fulfilled  at $\xb$. Then $\xb$ is strongly M-stationary if and
only if it is S-stationary.
\end{theorem}
\begin{proof}
The statement follows immediately  from the fact that under
LICQ(0) there exist exactly one MPEC working set and this set
fulfills $J_g=\bar I_g$, $J_G=\bar I^{0+}\cup\bar I^{00}$,
$J_H=\bar I^{+0}\cup\bar I^{00}$.
\end{proof}

We summarize the relations between the various stationarity concepts in the following picture.
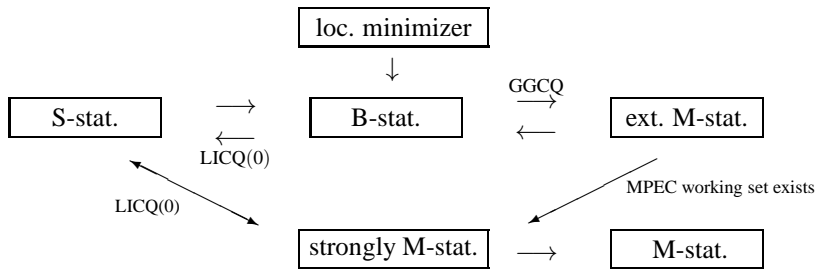
\unitlength1mm
\[\begin{picture}(100,40)
\put(0,20){\framebox(20,5){S-stat.}}
\put(20,20){\parbox{20mm}{\[\begin{array}{c}\longrightarrow\\\mathop{\longleftarrow}\limits_{{\rm LICQ(0)}}\end{array}\]}}
\put(40,20){\framebox(20,5){B-stat.}}
\put(60,22.5){\parbox{20mm}{\[\begin{array}{c}\mathop{\longrightarrow}\limits^{\rm
GGCQ}\\\longleftarrow\end{array}\]}}
\put(80,20){\framebox(20,5){ext. M-stat.}}
\put(38.5,2.5){\framebox(25,5){strongly M-stat.}}
\put(60,0){\parbox[b]{20mm}{\[\longrightarrow\]}}
\put(80,2.5){\framebox(20,5){ M-stat.}}
\put(38.5,32){\framebox(25,5){loc. minimizer}}
\put(50,28){$\downarrow$}
\put(24.5,12.5){\vector(2,-1){8.5}}
\put(24.5,12.5){\vector(-2,1){8.5}}
\put(14,10){\scriptsize LICQ(0)}
\put(82,12.5){\scriptsize MPEC working set exists}
\put(86,17){\vector(-2,-1){17}}
\end{picture}\]
We see from this picture, that S-stationarity also implies strong M-stationarity under the assumptions GGCQ and that one MPEC working set exists, which is much weaker that LICQ(0).

Using similar arguments  it can also be shown that under the
weaker condition {\em partial MPEC LICQ} \cite{Ye05} the concepts
of strongly M-stationarity and S-stationarity are equivalent.
However, for other conditions ensuring S-stationarity like the
{\em intersection property} \cite{FleKanOut07} or the condition
found in \cite{FuPa99}, the relation between strong M-stationarity
and S-stationarity is still unknown.

Finally we present an example where a local minimizer is strongly M-stationary  but not S-stationary.

\begin{example}[cf. \cite{SchSch00,FaLeyMun12}]
Consider
\begin{eqnarray*}\min_{x=(x_1,x_2,x_3)} f(x)&:=&x_1+x_2-x_3\\
g_1(x)&:=& -4x_1+x_3\leq 0,\\
g_2(x)&:=& -4x_2+x_3\leq 0,\\
-(G_1(x),H_1(x))&:=&-(x_1,x_2)\in Q_{\rm EC}.
\end{eqnarray*}
Then $\xb=(0,0,0)$ is  a local minimizer, GGCQ is fulfilled since all constraints are linear and therefore the multifunction  $M(x)=F(x)-\Omega$ is polyhedral and consequently metrically subregular at $(\xb,0)$ by Robinson's result \cite{Rob81}. Further it can be easily checked that $J_g=\{1,2\}$, $J_G=\{1\}$, $J_H=\emptyset$ constitutes a MPEC working set and, by taking the multipliers $\lambda_1^g=\frac 34$, $\lambda^g_2=\frac 14$, $\lambda_1^G=-2$, $\lambda_1^H=0$, we see that $\xb$ is strongly M-stationary. But, as pointed out in \cite{FaLeyMun12}, $\xb$ is not S-stationary.
\end{example}

\section*{Acknowledgments} The author is very grateful to
the referees for constructive comments that significantly improved the presentation. In particular, the proof of Theorem \ref{ThEssLocMin} could be considerably simplified.

\end{document}